\documentclass[letterpaper,10pt,reqno]{amsart}
\usepackage{indentfirst} 
\usepackage{amssymb}
\usepackage{mathtools} 
\usepackage{mathabx} 
\usepackage{amsthm}
\usepackage{thmtools}
\usepackage{enumitem} 
\usepackage[backref=page,colorlinks=true]{hyperref} 
\usepackage[usenames,dvipsnames]{xcolor} 
\usepackage{ifthen}
\usepackage{tikz}
\usetikzlibrary{decorations.pathreplacing}
\usepackage{tikz-cd}
  
\renewcommand*{\backref}[1]{}
\renewcommand*{\backrefalt}[4]{\tiny
  \ifcase #1 (\textbf{NOT CITED.})%
  \or    (Cited on page~#2.)%
  \else   (Cited on pages~#2.)%
  \fi}

\makeatletter

\def\MRbibitem{\@ifnextchar[\my@lbibitem\my@bibitem}

\def\mybiblabel#1#2{\@biblabel{{\hyperref{http://www.ams.org/mathscinet-getitem?mr=#1}{}{}{#2}}}}

\def\myhyperanchor#1{\Hy@raisedlink{\hyper@anchorstart{cite.#1}\hyper@anchorend}}

\def\my@lbibitem[#1]#2#3#4\par{%
  \item[\mybiblabel{#2}{#1}\myhyperanchor{#3}\hfill]#4%
  \@ifundefined{ifbackrefparscan}{}{\BR@backref{#3}}%
  \if@filesw{\let\protect\noexpand\immediate
    \write\@auxout{\string\bibcite{#3}{#1}}}\fi\ignorespaces%
}

\def\my@bibitem#1#2#3\par{%
  \refstepcounter\@listctr
  \item[\mybiblabel{#1}{\the\value\@listctr}\myhyperanchor{#2}\hfill]#3%
  \@ifundefined{ifbackrefparscan}{}{\BR@backref{#2}}%
  \if@filesw\immediate\write\@auxout
    {\string\bibcite{#2}{\the\value\@listctr}}\fi\ignorespaces%
}

\makeatother

\DeclareFontFamily{U} {MnSymbolA}{}
\DeclareFontShape{U}{MnSymbolA}{m}{n}{
   <-6> MnSymbolA5
   <6-7> MnSymbolA6
   <7-8> MnSymbolA7
   <8-9> MnSymbolA8
   <9-10> MnSymbolA9
   <10-12> MnSymbolA10
   <12-> MnSymbolA12}{}
\DeclareFontShape{U}{MnSymbolA}{b}{n}{
   <-6> MnSymbolA-Bold5
   <6-7> MnSymbolA-Bold6
   <7-8> MnSymbolA-Bold7
   <8-9> MnSymbolA-Bold8
   <9-10> MnSymbolA-Bold9
   <10-12> MnSymbolA-Bold10
   <12-> MnSymbolA-Bold12}{}
\DeclareSymbolFont{MnSyA} {U} {MnSymbolA}{m}{n}
 \DeclareFontFamily{U} {MnSymbolC}{}
\DeclareFontShape{U}{MnSymbolC}{m}{n}{
  <-6> MnSymbolC5
  <6-7> MnSymbolC6
  <7-8> MnSymbolC7
  <8-9> MnSymbolC8
  <9-10> MnSymbolC9
  <10-12> MnSymbolC10
  <12-> MnSymbolC12}{}
\DeclareFontShape{U}{MnSymbolC}{b}{n}{
  <-6> MnSymbolC-Bold5
  <6-7> MnSymbolC-Bold6
  <7-8> MnSymbolC-Bold7
  <8-9> MnSymbolC-Bold8
  <9-10> MnSymbolC-Bold9
  <10-12> MnSymbolC-Bold10
  <12-> MnSymbolC-Bold12}{}
\DeclareSymbolFont{MnSyC} {U} {MnSymbolC}{m}{n}

\DeclareMathSymbol{\top}{\mathord}{MnSyA}{219} 
\DeclareMathSymbol{\plus}{\mathord}{MnSyC}{20} 

\declaretheorem[numberwithin=section]{theorem}
\declaretheorem[sibling=theorem]{lemma}
\declaretheorem[sibling=theorem]{corollary}
\declaretheorem[sibling=theorem]{proposition}
\declaretheorem[sibling=theorem,style=definition]{definition}
\declaretheorem[sibling=theorem,style=remark]{example}

\declaretheorem[sibling=theorem,style=remark]{remark}
\declaretheorem[name=Acknowledgements, style=remark, numbered=no]{ack}

\hypersetup{bookmarksdepth = 3} 
\numberwithin{equation}{section}     

\setlist[enumerate,1]{label={\upshape(\alph*)},ref=\alph*}
\setlist[enumerate,2]{label={\upshape(\arabic*)},ref=\arabic*}

\newcommand{\R}{\mathbb{R}}
\newcommand{\Z}{\mathbb{Z}}
\newcommand{\N}{\mathbb{N}}

\newcommand{\F}{\mathcal{F}}

\newcommand{\M}{\mathcal{M}}\newcommand{\cO}{\mathcal{O}}
\newcommand{\cP}{\mathcal{P}}\newcommand{\cQ}{\mathcal{Q}}\newcommand{\cR}{\mathcal{R}}
\newcommand{\cU}{\mathcal{U}}

\newcommand{\Rnon}{\mathbb{R}_{\plus}}    
\newcommand{\Rpos}{\mathbb{R}_{\plus\plus}} 
\newcommand{\Mat}[2][]{\ifthenelse{\equal{#1}{}}{\R^{{#2}\times{#2}}}{\R^{{#1}\times{#2}}}}
\newcommand{\Man}[2][]{\ifthenelse{\equal{#1}{}}{\Rnon^{{#2}\times{#2}}}{\Rnon^{{#1}\times{#2}}}}
\newcommand{\Map}[2][]{\ifthenelse{\equal{#1}{}}{\Rpos^{{#2}\times{#2}}}{\Rpos^{{#1}\times{#2}}}}

\renewcommand{\epsilon}{\varepsilon}
\renewcommand{\phi}{\varphi}
\renewcommand{\setminus}{\smallsetminus}

\begin{document}

\title{The space of invariant measures for countable Markov shifts}
\date{\today}


\begin{thanks}
{G.I.\ was partially supported by CONICYT PIA ACT172001 and by Proyecto Fondecyt 1190194.}
\end{thanks}

\author[G.~Iommi]{Godofredo Iommi}
\address{Facultad de Matem\'aticas,
Pontificia Universidad Cat\'olica de Chile (PUC), Avenida Vicu\~na Mackenna 4860, Santiago, Chile}
\email{giommi@mat.puc.cl}
\urladdr{\url{http://http://www.mat.uc.cl/~giommi/}}

 \author[A.~Velozo]{Anibal Velozo}  \address{Department of Mathematics, Yale University, New Haven, CT 06511, USA.}
\email{anibal.velozo@yale.edu }
\urladdr{\url{https://gauss.math.yale.edu/~av578/}}

\maketitle

\begin{abstract}
It is well known that the space of invariant  probability measures for transitive sub-shifts of finite type is a Poulsen simplex. In this article we prove that in the non-compact setting, for a large family of transitive countable Markov shifts, the space of invariant sub-probability measures is a Poulsen simplex and that its extreme points are the ergodic invariant probability measures together with the zero measure. In particular we obtain that the space of invariant probability measures  is a Poulsen simplex minus a vertex and the corresponding convex combinations. Our results apply to finite entropy non-locally compact transitive countable Markov shifts and to every locally compact transitive countable Markov shift. In order to prove these results we introduce a topology on the space  of measures that generalizes the vague topology to a class of non-locally compact spaces, the topology of convergence on cylinders. We also prove analogous results for suspension flows defined over countable Markov shifts. 
\end{abstract}
\vspace{1cm}

\section{Introduction}

Ever since the work of Parthasarathy \cite{par} and  Oxtoby \cite{o} in the early 1960s a great deal of attention has been paid to the problem of describing the space of invariant probability measures of a dynamical system. Remarkable results have been obtained relating the geometry of the space with the dynamical properties of the system. A result by Downarowicz \cite{d} states that  for every Choquet simplex $K$ there exists a minimal sub-shift $(X, T)$ for which the space of invariant probability measures $\M(X,T)$ is affinely homemorphic to $K$. In this article we will be interested in a very special Choquet Simplex. 
\begin{definition}
A  metrizable convex compact Choquet simplex with at least two points $K$ is a \emph{Poulsen Simplex} if its extreme points are dense in $K$.
\end{definition}
The first example of such a simplex was constructed by Poulsen \cite{pou} in $1961$. It was later shown by  Lindenstrauss,  Olsen and  Sternfeld \cite[Theorem 2.3]{los} that the Poulsen simplex is unique up to affine homemorphism. This simplex enjoys remarkable properties. For example, as proved in  \cite[Section 3]{los}, the set of extreme points in the Poulsen simplex is path connected.
The relation of this simplex with dynamical systems directly follows from the seminal work of Sigmund \cite{si}, see also \cite{si1,si2}. Indeed,  if $(\Sigma,\sigma)$ is a transitive sub-shift of finite type with infinitely many elements  then $\M(\Sigma, \sigma)$ is affinely homeomorphic to the Poulsen simplex. Note that the extreme points in this setting correspond to the ergodic measures. 

This article describes the space of invariant measures for transitive one-sided countable state Markov shifts (see Section \ref{sec:cms} for precise definitions). The major difference with previous work on the subject is that the phase space is no longer compact and therefore the escape of mass phenomenon has to be taken into account.  Notions of convergence in the space of measures are required to describe loss of mass. Indeed, the weak* topology preserves the total mass of the space, thus it can not capture the escape of mass. For locally compact spaces the space of invariant measures can be endowed with the vague topology; in this context it is possible for mass to be lost. We introduce a new notion of convergence in the space of 
 measures, the so called \emph{topology of convergence on cylinders},  that generalizes the vague topology.  This notion of convergence does not require the underlying space to be locally compact.

In the non-compact setting the space of invariant probability measures is not necessarily compact. The lack of compactness of the space of invariant probability measures    is a major difficulty in the development of the corresponding ergodic theory:  in many arguments it is natural to take limits of invariant measures and it is  important to know that the limiting object is indeed a measure. 
We stress that this is a very subtle phenomenon, it could  happen that  for topologies that naturally generalize the weak* topology the limit of a sequence of invariant probability measures is not a countably additive measure. In this paper we will compactify the space of invariant probability measures for a large family of countable Markov shifts, including a wide range of non-locally compact shifts.  As we will see, our compactification is strongly related to the escape of mass phenomenon.

 For completeness we will briefly describe the topology on the space of invariant sub-probability measures we will focus on in this work. Let $(\Sigma,\sigma)$ be a transitive countable Markov shift and $\M_{\le1}(\Sigma,\sigma)$ the space of $\sigma$-invariant sub-probability measures on $\Sigma$ (for precise definitions we refer the reader to Sections \ref{sec:cms} and \ref{sec:top}). We say that $(\mu_n)_n\subset \M_{\le1}(\Sigma,\sigma)$ converges on cylinders to $\mu$ if $$\lim_{n\to\infty}\mu_n(C)=\mu(C),$$
for every cylinder $C\subset \Sigma$. This notion of convergence induces a topology, the topology of convergence on cylinders. We prove that this topology is  metrizable (see Proposition \ref{metriz}). For general facts about the  topology of convergence on cylinders we refer  the reader to Section \ref{sec:cyl}. 

We consider a large class of countable Markov shifts that  satisfy the so called $\mathcal{F}-$property (see Definition \ref{def:F}).
This include locally compact and finite entropy non-locally compact countable Markov shifts.   The  $\mathcal{F}-$property essentially rules out the possibility of having infinitely many periodic orbits of a given length that intersect a fixed cylinder. One of the main results of this work is

\begin{theorem} \label{thm:po}
Let $(\Sigma, \sigma)$ be a transitive countable Markov shift satisfying the $\F-$property. Then the space of invariant sub-probability measures $\M_{\le 1}(\Sigma,\sigma)$ endowed with the topology of convergence on cylinders is affine homeomorphic to the Poulsen simplex. In particular $\M_{\le1}(\Sigma,\sigma)$ is compact with respect to the topology of convergence on cylinders. 
\end{theorem}

The topology of convergence on cylinders restricted to $\M(\Sigma,\sigma)$ coincides with the weak* topology (see Lemma \ref{equivtop}). Theorem \ref{thm:po} has the following corollary.

\begin{corollary}
Let $(\Sigma, \sigma)$ be a transitive countable Markov shift satisfying the $\F-$property. Then the space of invariant probability measures $\M(\Sigma, \sigma)$ endowed with the weak* topology is affine homeomorphic to the Poulsen simplex minus a vertex and all of its convex combinations.
\end{corollary}

In order to prove Theorem \ref{thm:po} we will need to prove three key properties: (1) there exists a sequence of ergodic measures converging on cylinders to the zero measure,  (2) every sequence of periodic measures has an accumulation point (in the topology of convergence on cylinders) which is a countably additive measure and (3) the set of periodic measures is weak* dense in $\M(\Sigma, \sigma)$. 
While point (3) is fairly standard and uses shadowing and closing properties of the shift, the other two are more subtle. Indeed, a combinatorial assumption is required on $(\Sigma,\sigma)$ for these properties to hold (hence the $\mathcal{F}$-property assumption). 

It worth pointing out that  Theorem \ref{thm:po} is optimal for the topology of convergence on cylinders. More precisely, if $(\Sigma,\sigma)$ does not satisfy the $\F-$property, then $\M(\Sigma,\sigma)$ contains a sequence of periodic measures which converges on cylinders to a finitely additive measure which is not countably additive (see Proposition \ref{prop:noF}).  In particular if we want to compactify $\M(\Sigma,\sigma)$ we must give up the convergence on all cylinders (which does not seem reasonable) or modify the topology in a more substantial way.

We also study suspension flows defined over countable Markov shifts. These are continuous time dynamical systems defined over non-compact spaces. The suspensions we consider are constructed over arbitrary countable Markov shifts $(\Sigma, \sigma)$ and for  roof functions $\tau:\Sigma\to\R$ belonging to a class that we denote by $\mathcal{R}$ (see Definition \ref{def:R}).
If $\tau$ is bounded away from zero there is a one--to--one correspondence between the space of invariant probability measures for the flow, which we denote by $\M(\Sigma,\sigma,\tau)$, and $\M_\tau=\{\mu\in \M(\Sigma,\sigma):\int \tau d\mu<\infty\}$. The space of  sub-probability measures invariant by the suspension flow is denoted by $\M_{\le1}(\Sigma,\sigma,\tau)$.  In Section \ref{sec:sus} we define a topology on $\M_{\le1}(\Sigma,\sigma,\tau)$, the \emph{topology of convergence on cylinders for the suspension flow}, that shares many properties with the topology of convergence on cylinders on $\M_{\le1}(\Sigma,\sigma)$ (see Definition \ref{def:ccs}). 

The class of suspension flows that we study includes a wide range of symbolic models for geometric systems. For example, the symbolic model of the geodesic flow over the modular surface satisfies all of our assumptions.  In this context we prove,

\begin{theorem}\label{thm:po2} Let $(\Sigma,\sigma)$ be a transitive countable Markov shift and $\tau\in \cR$. Then the space of invariant sub-probability measures of the suspension flow $\M_{\le 1}(\Sigma,\sigma,\tau)$, endowed with the topology of convergence on cylinders for the suspension flow, is affine homeomorphic to the Poulsen simplex. In particular $\M_{\le 1}(\Sigma,\sigma,\tau)$ is compact with respect to the topology of convergence on cylinders.
\end{theorem}


Some countable Markov shifts without the $\F-$property are particularly important (for instance the full shift, or shifts with the BIP property \cite{sa2}), and we do want to have some understanding on their spaces of invariant probability measures. The work done in Section \ref{sec:sus} and our auxiliary potential $\tau$ allows us to regain control in this setting. Indeed, in this general context we are able to describe the set of invariant probability measures for which the function $\tau$ is integrable (see Lemma \ref{lem:medi} and Theorem  \ref{susp:p2}).
 
We stress that the compactifications constructed in this paper have several interesting applications to the thermodynamic formalism of countable Markov shifts. For instance, in joint work with M. Todd \cite{itv}, we consider finite entropy countable Markov shifts $(\Sigma, \sigma)$ and study the behaviour of the measure theoretic entropy of a sequence $(\mu_n)_n\subset\M(\Sigma,\sigma)$. By the results in this paper there exists a measure $\mu\in \M_{\le1}(\Sigma,\sigma)$ and a sub-sequence of $(\mu_n)_n$ which converges to $\mu$ in the cylinder topology. This property is used  to establish stability results for the measure of maximal entropy, relate the escape of mass with the entropy of the system and to prove upper-semi continuity of the entropy map. In the compact setting there is no escape of mass and the properties of the entropy map and the measure of maximal entropy are classical \cite[Chapter 8]{wa}, but in the non-compact case new ideas were needed. We also mention that the results in \cite{itv} can be pushed even further to include potentials, this is discussed in \cite{v} by the second author.

Finally, we remark that over the last few years countable Markov shifts have been used to code relevant parts of the dynamics for a wide range of dynamical systems. For example, it was shown by Sarig \cite{sa} that countable Markov partitions can  be constructed for positive entropy diffeomorphisims defined on compact surfaces. The corresponding symbolic coding captures positive entropy measures.  These results  have recently been used to prove that $C^{\infty}$ surface diffeomoprhisms of positive entropy have at most finitely many measures of maximal entropy \cite{bcs}. In a different direction, countable Markov partitions have been constructed  for Sinai and Bunimovich billiards, this has allowed for the proof of lower bounds on the number of periodic orbits of a given period \cite{lm}. Based on the work of Sarig, countable Markov partitions have been constructed for  large classes of dynamical systems.  Our results apply not only to all  countable Markov shifts obtained as symbolic codings of these systems, but also to several symbolic models that are not locally-compact. For instance, symbolic codings of interval maps having a parabolic fixed point \cite{h, mp, sa1} or loop systems \cite{bbg}.

\begin{ack}
We would like to thank Mike Todd for a wealth of relevant and interesting comments on the subject of this article. We would also like to thank the referee for many useful comments and suggestions. This paper was initiated while the second author was visiting the first author at Pontificia Universidad Cat\'olica de Chile. The second author would like to thank the dynamics group at PUC for making his visit very stimulating. He would also like to thank Richard Canary for his invitation to participate of `Workshop on Groups, Geometry and Dynamics' held in  Universidad de la Rep\'ublica, where an important part of this work was prepared. 
\end{ack}


\section{Countable Markov shifts}\label{sec:cms}
In this section we define the dynamical systems that will be studied throughout the article. Let $B$ be a transition matrix defined on the alphabet of natural numbers. That is, the entries of the matrix $B=B(i,j)_{\N   \times \N}$  are zeros and ones (with no row and no column made entirely of zeros). The one-sided countable state Markov shift $(\Sigma, \sigma)$ defined by the matrix $B$ is the set
\begin{equation*}
\Sigma := \left\{ (x_n)_{n \in \N} : B(x_n, x_{n+1})=1  \text{ for every } n \in \N   \right\}, 
\end{equation*}
together with the shift map $\sigma: \Sigma  \to \Sigma $ defined by $\sigma(x_1, x_2, \dots)=(x_2, x_3,\dots)$.
In what follows we will simply call $(\Sigma, \sigma)$,  \emph{countable Markov shift}.
For $(a_{1}, \dots ,a_{n}) \in \N^n$, we define a cylinder set $[a_1 \dots a_{n}]$ of length $n$ by 
\begin{equation*}
[a_1 \dots a_{n}]:= \left\{x \in \Sigma: x_j=a_j \text{ for } 1 \le j \le n \right\}.
\end{equation*}
We endow $\Sigma$ with the topology generated by cylinder sets. This is a metrizable  non-compact  space. Indeed, let $\tilde{d}:\Sigma \times \Sigma \to \R$ be the function defined by  
\begin{equation} \label{metric}
\tilde{d}(x,y):=
\begin{cases}
1 & \text{ if } x_0\ne y_0; \\
2^{-k} & \text{ if  } x_i=y_i \text{ for  } i \in \{0, \dots , k-1\} \text{ and } x_k \neq y_k; \\
0 & \text{ if } x=y.
\end{cases}
\end{equation} 
The function $\tilde{d}$ is a metric and it generates the same topology as that of the cylinders sets. 

A countable Markov shift  defined by the transition  matrix    
$B=B(i,j)_{\N  \times \N }$  
is locally compact if and only if for every $i \in \N$ we have $\sum_{j=1}^{\infty} B(i,j) < \infty$ (see \cite[Observation 7.2.3]{ki}). 

An \emph{admissible word} is a word $w={\bf a_1...a_n}$, where ${\bf a_i}\in \N$ and $[a_1,...,a_n]$ is non-empty. To emphasize the difference between admissible words and points in $\Sigma$ we use bold letters for admissible words. 

\begin{definition}
A countable Markov shift $(\Sigma, \sigma)$ is \emph{transitive} if for every  open sets $U, V \subset \Sigma$ there exists $n \in \N$ such that $U \cap \sigma^{-n} V \neq \emptyset$. Similarly, a countable Markov shift $(\Sigma, \sigma)$ is \emph{topologically mixing} if for every open sets $U, V \subset \Sigma$ there exists $N(U,V) \in \N$ such that for every $n >N(U,V)$ we have $U \cap \sigma^{-n} V \neq \emptyset$. \end{definition}

Let $\phi:\Sigma\to\R$ be a function. We define $var_n(\phi)=\sup_{x,y} |\phi(x)-\phi(y)|$, where the supremum runs over points $x$ and $y$ satisfying $\tilde{d}(x,y)\le 2^{-n}$. Observe that a function $\phi$ is uniformly continuous if and only if $var_n(\phi)$ goes to zero as $n$ goes to infinity. A potential $\phi$ has summable variations if $\sum_{k=2}^\infty var_k(\phi)$ is finite. 

In the late 1960s Gurevich \cite{gu1,gu2} introduced a suitable notion of entropy in this setting. Note that since the space $\Sigma$ is not compact the classical definition of topological entropy obtained by means of $(n,\epsilon)$-separated sets (see \cite[Chapter 7]{wa}) depends upon the metric. That is, two equivalent metrics can yield different numbers.  Since the entropy of an invariant measure depends only on the Borel structure and not on the metric, this is a major problem if the entropy is to satisfy a variational principle. Gurevich introduced the following notion of entropy:
\begin{equation*}
h(\sigma):= \limsup_{n \to \infty} \frac{1}{n} \log \sum_{x:\sigma^n x =x} 1_{[a]}(x),
\end{equation*}
where $a \in \N$ is an arbitrary symbol and $1_{[a]}$ is the characteristic function of the cylinder $[a]$.
Gurevich  proved that this value is independent of the symbol $a$ if $(\Sigma, \sigma)$ is transitive and that the limit exists if $(\Sigma,\sigma)$ is topologically mixing. Moreover, he also proved that this notion of entropy is the correct one in the sense that it satisfies the variational principle. That is
\begin{equation*}
h(\sigma)= \sup \left\{h(\mu) : \mu \in \mathcal{M}( \Sigma, \sigma) \right\},
\end{equation*}
where $h(\mu)$ is the entropy of the invariant measure $\mu$ (see \cite[Chapter 4]{wa}) and  $\mathcal{M}( \Sigma, \sigma)$ is the space of invariant probability measures.

\section{Topologies in the space of measures}\label{sec:top}

In this section we recall definitions and properties of the weak* and the vague topologies and define a new notion of convergence in the space of probability measures on $\Sigma$, namely the topology of convergence on cylinders. It is with respect to these three topologies that we will describe the space of invariant probability measures $ \mathcal{M}( \Sigma, \sigma)$. It  is worth emphasizing that the weak* topology does not allow escape of mass, but the vague topology and the topology of convergence on cylinders do allow it. 


\subsection{The weak* topology}

Let $(X,\rho)$ be a metric space. We denote by $C_b(X)$ the space of bounded continuous function $f:X\to \R$. We endow $C_b(X)$ with the $C^0$-topology. This is the topology induced by the norm $\| f \|=\sup_{x\in X}|f(x)|$. It is a standard fact that $C_b(X)$ is a Banach space. Denote by $\M(X)$ the set of Borel probability measures on the metric space $(X,\rho)$. Our first notion of convergence in this set is the following,

\begin{definition} \label{def:wc}
A sequence of probability measures $(\mu_n)_n\subset \M(X)$ converges to a measure $\mu$ in the weak* topology if for every $f \in C_b(X)$ we have 
\begin{equation*}
\lim_{n \to \infty} \int f d \mu_n = \int f d \mu.
\end{equation*}
\end{definition}

\begin{remark}
Note that in this notion of convergence we can replace the set of test functions (bounded and continuous) by the space of bounded uniformly continuous functions (see \cite[8.3.1 Remark]{bg}) or by the space of bounded Lipschitz functions (see \cite[Theorem 13.16 (ii)]{kl}). That is, if for every bounded uniformly continuous (or bounded Lipschitz) function $f:X\to \R$ we have 
\begin{equation*}
\lim_{n \to \infty} \int f d \mu_n = \int f d \mu,
\end{equation*}
then the sequence $(\mu_n)_n$ converges in the weak* topology to $\mu$.
\end{remark}

The  weak* topology is the coarsest topology such that for every $f \in C_b(X)$ the map  $\mu \to \int f \ d\mu$, with $\mu \in \M(X)$,  continuous. The following classical result characterizes weak* convergence (see \cite[Theorem 2.1]{bi}).

\begin{proposition}[Portmanteau Theorem] \label{port}
Let $(\mu_n)_n, \mu$ be probability measures on $X$. The following statements are equivalent.
\begin{enumerate}
\item The sequence $(\mu_n)_n$ converges to $\mu$ in the weak* topology.
\item For every open set $O \subset X$, the following holds $\mu(O) \leq  \liminf_{n \to \infty} \mu_n(O)$.
\item  For every closed set $C \subset X$, the following holds $\mu(C) \geq  \limsup_{n \to \infty} \mu_n(C)$.
\item\label{port4}  For every set $A \subset X$ such that $\mu(\partial A)=0$, the following holds $\mu(A) = \lim_{n \to \infty} \mu_n(A)$.
\end{enumerate}
\end{proposition}

A relevant feature of the weak* convergence is that there is no loss of mass since the constant function equal to one belongs to $C_b(X)$. That is,
\begin{remark}
If the sequence of probability measures $(\mu_n)_n$ converges in the weak* topology to $\mu$ then $\mu$ is also a probability measure.
\end{remark}
Also note that if the space $(X,\rho)$ is compact then the space of Borel probability measures, $\M(X)$,  is also compact with respect to the weak* topology (see \cite[Theorem 6.5]{wa}). An interesting fact is that if $(X,\rho)$ is separable metric space then $\M(X)$ can be metrized as a separable metric space (see \cite[Theorem 6.2]{pa}).
Actually, there exists an explicit metric that generates the weak* topology and for which $\M(X)$ is separable if $X$ is separable. This is the so called Prohorov metric (see \cite[pp.72-73]{bi}). Therefore, if $X$ is a separable metric space then so is $\M(X)$ despite the fact that the space $C_b(X)$ might not be separable. Actually, if $(\Sigma, \sigma)$ is a non-compact countable Markov shift then $C_b(\Sigma)$ is not separable, as below.

\begin{remark}
Let $(\Sigma, \sigma)$ be a countable (non-compact) Markov shift then $C_b(\Sigma)$ is not separable. Indeed, let $(x_n)_n$ be a fixed sequence of elements of $\Sigma$ with $x_n \in [n]$, where $[n]:=\{(y_1, y_2 ,\dots) \in \Sigma : y_1 =n\}$. Define the set
\begin{equation*}
\mathcal{L}:= \left\{ \phi \in C_b(\Sigma) : \phi(x_n)=0 \text{ or } \phi(x_n)= 1, \text{ for every } n \in \N \right\}. 
\end{equation*}
Note that $\mathcal{L}$ contains uncountably many elements. For every countable subset $\mathcal{C} \subset C_b(\Sigma)$ there exists $\phi \in \mathcal{L}$ such that for every $\psi \in \mathcal{C}$ we have
\begin{equation*}
\| \phi - \psi\| \geq \frac{1}{2},
\end{equation*}
hence $C_b(\Sigma)$ is not separable.
\end{remark}


Let $T:(X,\rho) \to (X,\rho)$ be a continuous dynamical system defined on a metric space. We denote by $\M (X, T)$ the space of $T-$invariant probability measures. In the following lemma we collect relevant information regarding the structure of $\M (X, T)$.

\begin{lemma} \label{lem_inv}
Let $T:(X,\rho) \to (X,\rho)$ be a continuous dynamical system defined on a metric space, then
\begin{enumerate}
\item The space $\M (X, T)$, as a subset of $\M(X)$, is closed in the weak* topology (\cite[Theorem 6.10]{wa}).
\item  If $X$ is compact then so is $\M(X,T)$  with respect to the weak* topology (see \cite[Theorem 6.10]{wa}).
\item The space $\M (X, T)$ is a convex set for which its extreme points are the ergodic measures  (see \cite[Theorem 6.10]{wa}). It is actually a Choquet simplex (each measure is represented in a unique way as a generalized convex combination of the ergodic measures  \cite[p.153]{wa}).
\end{enumerate}
\end{lemma}

\begin{definition} Let $T:(X,\rho) \to (X,\rho)$ be a continuous dynamical system defined on a metric space. We denote by $\M_e(X,T)$ the set  of ergodic $T$-invariant  probability measures. An ergodic  measure $\mu \in \M_e( X, T)$ supported on a periodic orbit will  be called a \emph{periodic measure}. Denote by $\M_p(X,T)$  the set of periodic measures. 
\end{definition}

The next result was obtained by Coud\'ene and  Schapira \cite[Section 6]{cs} as a consequence of shadowing and the Anosov closing Lemma.
\begin{theorem} \label{lem:cs}
Let $(\Sigma, \sigma)$ be a transitive countable Markov shift then  $\M_p(\Sigma,\sigma)$ is dense in  $\M(\Sigma,\sigma)$ with respect to the weak* topology.
\end{theorem}
%

\subsection{The vague topology}

Let $(X,\rho)$ be a locally compact metric space. Denote by $M_{\le 1}(X)$ the set of Borel non-negative measures on $X$ such that $\mu(X) \leq 1$. The set of continuous functions of compact support, that is continuous functions $f:X \to \R$ for which the closure of the set $\{x \in X : f(x) \neq 0 \}$ is compact, will be denoted by $C_c(X)$. Note that $C_c(X) \subset C_b(X)$. We will  consider the following notion of convergence.
 
\begin{definition}
A sequence $(\mu_n)_n\subset \M_{\le1}(X)$ converges to  $\mu\in \M_{\le1}(X)$ in the vague topology if for every   $f \in C_c(\Sigma)$ we have 
\begin{equation*}
\lim_{n \to \infty} \int f d \mu_n = \int f d \mu.
\end{equation*}
\end{definition}


The vague topology  is the coarsest topology on $\M_{\le 1}(X)$ such that for every $f$ continuous and of compact support, the map $\mu \to \int f d\mu$ is continuous. We stress that the total mass is not necessarily preserved in the vague topology. A sequence of probability measures can converge in the vague topology to a non-negative measure of total mass less or equal to one.  If $X$ is compact then $C_c(X)=C_b(X)$ and therefore the vague topology coincides with the weak* topology. We collect the following results,

\begin{remark} \label{rem_va}
Let $(X,\rho)$ be a metric space. Note that  the weak* topology extends to $\M_{\leq 1}(X)$.
\begin{enumerate}
\item If $X$ is compact then $\M_{\leq 1}(X)$ is compact with respect to the weak* topology (see \cite[Corollary 13.30]{kl}).
\item If $X$ is a locally compact separable metric space then $\M_{\leq 1}(X)$ is compact with respect to the vague topology (see \cite[Corollary 13.31]{kl}) and metrizable (see \cite[13.4.2]{di}).
\item Let $X$ be a locally compact separable metric space. The sequence $(\mu_n)_n$ converges vaguely to $\mu$ and $\lim_{n \to \infty} \mu_n(X)= \mu(X)$ if and only if  $(\mu_n)_n$ converges in weak* topology to $\mu$   (see \cite[Theorem 13.16]{kl}).
\item Let $X$ be a locally compact separable metric space. The sequence $(\mu_n)_n$ converges vaguely to $\mu$ and the sequence $(\mu_n)_n$ is tight if and only  if  $(\mu_n)_n$ converges in weak* topology to $\mu$   (see \cite[Theorem 13.35]{kl}).
\end{enumerate}
\end{remark}

Let $T:(X,\rho) \to (X,\rho)$ be a continuous dynamical system defined on a metric space.
If $(\mu_n)_n \subset \M(X,T)$ is a sequence of $T-$invariant probability measures that converges in the vague topology to a non zero measure $\mu$, then the normalized measure $\mu( \cdot) / \mu(X)$ 
 is a $T-$invariant probability.  We  call the measure $\mu$ an invariant sub-probability and denote by  $\M_{\le 1}(X,T)$ the space of $T-$invariant sub-probability measures. Observe that the zero measure belongs to $\M_{\le 1}(X,T)$.

\begin{lemma} \label{lem:comp}
Let $T:(X,\rho) \to (X,\rho)$ be a continuous dynamical system defined on a metric space, then
\begin{enumerate}
\item  The space  $\M(X,T)$  is a closed subset of $\M_{\le 1}(X,T)$ in the weak* topology.
\item If $X$ is a locally compact separable metric space then the space $\M_{\le 1}(X,T)$ is compact in the vague topology.
\end{enumerate}
\end{lemma}

\begin{proof}
The first claim is a consequence of Lemma \ref{lem_inv}, while the second follows from Remark \ref{rem_va}.
\end{proof}

\begin{proposition} \label{lem_ext}
Let $T:(X,\rho) \to (X,\rho)$ be a continuous dynamical system defined on a metric space. Then the space $\M_{\le 1} (X, T)$ is a convex set and its extreme points are the ergodic probability measures and the zero measure.
\end{proposition}

\begin{proof} The convexity of the space $M_{\le 1} (X, T)$ is direct. Note that every invariant sub-probability $\mu \in \M_{\le 1} (X, T)$ with $0< \mu(X)<1$, is  the convex combination of a measure in $\M(X,T)$ and the zero measure. Also note that the zero measure is not the convex combination of any set of positive measures. The result then follows from Lemma \ref{lem_inv}.
\end{proof}

%
%
%

\subsection{The topology of convergence on cylinders} \label{sec:cyl} Several relevant countable Markov shifts are not locally compact. Therefore a good notion of convergence in the space of sub-probabilities is required in this setting. The vague topology is of no use in the non-locally compact setting since in this case the space $C_c(X)$  might be empty, as below.

\begin{remark}
If $\Sigma$ is a non-locally compact transitive countable Markov shift then $C_c(\Sigma)= \emptyset$. Indeed, if $K \subset \Sigma$ is a compact set then it must have empty interior (see \cite[Observation 7.2.3 (iv)]{ki}).  Therefore $K= \partial K$, that is, the compact set equals its topological boundary. If $f \in C_c(\Sigma)$ then consider the open set
\begin{equation*}
A:= \left\{x \in \Sigma : f(x) \neq 0 \right\} = f^{-1}\left(	\R \setminus \{0\}	\right).
\end{equation*}
Note that the support of $f$ is $K=\overline{A}$, that is the closure of the set $A$. But since $f \in C_c(\Sigma)$ the set $K$ is compact. This means that the compact set $K$ contains an open set $A$, contradicting the fact that $K= \partial K$. Therefore $C_c(\Sigma) = \emptyset$. 

\end{remark}

We now define the notion of convergence--that generalizes the vague topology to the non-locally compact setting--which is the main topic of this work.

\begin{definition} \label{def:topo.cyl} Let $(\Sigma,\sigma)$ be a countable Markov shift and $(\mu_n)_n,  \mu$ invariant sub-probability measures. We say that a sequence $(\mu_n)_n$ \emph{converges on cylinders} to $\mu$ if $\lim_{n\to\infty}\mu_n(C)=\mu(C)$, for every cylinder $C\subset \Sigma$.
The topology on $\M_{\le 1}(\Sigma)$  induced by this convergence is called the \emph{topology of convergence on cylinders}. 
\end{definition}

We emphasize that this notion of convergence induces a topology because the collection of cylinders is countable and it is a basis for the topology on $\Sigma$. For brevity we will frequently say that $(\mu_n)_n\subset \M_{\le1}(\Sigma)$ \emph{converges on cylinders} to $\mu$ if the sequence converges in the topology of convergence on cylinders. 
%

\begin{proposition}\label{metriz}
The \emph{topology of convergence on cylinders} on $\M_{\le 1}(\Sigma)$  is metrizable.
\end{proposition}

\begin{proof}
Consider the metric 
\begin{align}\label{metric} d(\mu,\nu)=\sum_{n\ge1} \frac{1}{2^n}|\mu(C_n)-\nu(C_n)|,\end{align}
where $(C_n)_n$ is some enumeration of the cylinders on $\Sigma$. Note that $d(\mu,\nu)=0,$ if and only if $\mu=\nu$. Indeed, if  $d(\mu,\nu)=0$ then $\mu(C)=\nu(C)$, for every cylinder $C$. By the outer regularity of Borel measures on a metric space we conclude that this is equivalent to say that $\mu=\nu$.  Symmetry is clear  and the triangle inequality follows directly from the triangle inequality in $\R$. It is clear from the definition of $d$ that it induces the desired notion of convergence on $\M_{\le 1}(\Sigma)$, that is, it generates the  topology of convergence on cylinders.
\end{proof}
It worth pointing out that since $C_b(\Sigma)$ is not separable  we can not endow the weak* topology with a metric like $d$. 

\begin{remark} \label{rem:escape}
Note that the topology of convergence on cylinders, like the vague topology,  allows for mass to escape. Indeed, let $\Sigma $ be the full shift on $\N$, which is not locally compact. Denote by $\delta_{\overline{n}}$ the atomic measure supported on the point $\overline{n}:=(n, n, n, \dots)$. Then the sequence $(\delta_{\overline{n}})_n$ converges in the topology of convergence on cylinders to the zero measure.
\end{remark}

Despite Remark \ref{rem:escape} the topology of convergence on cylinders is closely related to the weak* topology. If there is no loss of mass both notions  coincide.

\begin{lemma}\label{equivtop} Let $(\Sigma,\sigma)$ be a countable Markov shift,  $\mu$ and $(\mu_n)_n$ be probability measures on $\Sigma$. The following assertions are equivalent.
\begin{enumerate}
\item The sequence $(\mu_n)_n$ converges in the weak* topology to $\mu$.
\item The sequence $(\mu_n)_n$ converges on cylinders to $\mu$. 
\end{enumerate}
\end{lemma}

\begin{proof}
First assume that  $(\mu_n)_n$ converges in the weak* topology to $\mu$. Define $f:=1_C\in C_b(\Sigma)$, where $1_C$ is the characteristic function of a cylinder. The weak* convergence implies that $\lim_{n\to\infty} \int fd\mu_n=\int fd\mu$, which is equivalent to say that 
$$\lim_{n\to\infty} \mu_n(C) =\mu(C).$$
Since the cylinder $C$ was chosen arbitrarily we conclude that $(\mu_n)_n$ converges on cylinders to $\mu$. Now assume that $(\mu_n)_n$ converges on cylinders to $\mu$. Observe that an open set $\cO$ can be written as a countable union of disjoint cylinders, say $\cO=\bigcup_{k\ge 1}C_k$. Therefore $$\liminf_{n\to\infty}\mu_n(\cO)\ge \liminf_{n\to\infty}\mu_n\left(\bigcup_{k=1}^M C_k \right)=\mu\left(\bigcup_{k=1}^M C_k\right),$$
for every $M$. We conclude that $$\liminf_{n\to\infty}\mu_n(\cO)\ge \mu(\cO).$$ 
Proposition \ref{port} implies that $(\mu_n)_n$ converges in the weak* topology to $\mu$. 
\end{proof}

We will now prove that the topology of convergence on cylinders generalizes the vague topology. More precisely, on locally compact countable Markov shifts both topologies coincide. 

\begin{lemma} \label{lem_loc_cyl}
Let $(\Sigma,\sigma)$ be a locally compact countable Markov shift and $\mu , (\mu_n)_n \in \M_{\le 1}(\Sigma)$. The following assertions are equivalent.
\begin{enumerate}
\item The sequence $(\mu_n)_n$ converges in the vague topology to $\mu$.
\item The sequence $(\mu_n)_n$ converges on cylinders to $\mu$. 
\end{enumerate}
\end{lemma}

\begin{proof}
First note that if $\Sigma$ is locally compact then every cylinder is a compact set (see \cite[Observation 7.2.3]{ki}). Assume that  $(\mu_n)_n$ converges in the vague topology to $\mu$. Let $C \in \Sigma$ be a cylinder set, then the characteristic function of $C$, denoted by $1_C$ belongs to $C_c(\Sigma)$. Thus,
\begin{equation*}
\lim_{n \to \infty} \int 1_C d \mu_n = \int 1_C d \mu.
\end{equation*}
Therefore, the sequence $(\mu_n)_n$ converges on cylinders to $\mu$.

Suppose now that $(\mu_n)_n$ converges on cylinders to $\mu$ and let $f\in C_c(\Sigma)$. We will prove that $\lim_{n\to\infty}\int fd\mu_n=\int fd\mu$. The function $f$ is uniformly continuous, in particular for every $\epsilon>0$ there exists $n=n(\epsilon)\in\N$ such that $var_n(f)\le \epsilon$.  Since the support of $f$ is a compact set there exists  $M \in \N$ such that  it  is contained in $\bigcup_{i=1}^M[i]$.
By the locally compactness of $\Sigma$  there are finitely many (non-empty) cylinders of length $n$, that we denote by  $(C_i)_{i=1}^q$,  intersecting $\bigcup_{i=1}^M[i]$.  Note that if a cylinder of length $n$ intersects a cylinder of length one then it is contained in it, therefore $\bigcup_{i=1}^q C_i = \bigcup_{i=1}^M[i]$. We now define a locally constant function $\tilde{f}:\Sigma\to\R$ that approximates $f$. For every $i \in \{1,\dots, q\}$  choose a point $x_i\in C_i$ and let $\tilde{f}: \Sigma \to \R$ be the function defined by
\begin{equation*}
\tilde{f}(x):=
\begin{cases}
f(x_i) & \text{ if } x \in C_i, \text{ for }  i \in \{1,\dots, q\}; \\
0 & \text{ if } x \notin \bigcup_{i=1}^q C_i.
\end{cases}
\end{equation*}
By construction the function $\tilde{f}$ is locally constant depending  only on  the first $n$ coordinates, thus $var_n(\tilde{f})=0$. Moreover, it is zero on the complement of the set $\bigcup_{i=1}^M[i]$. In particular $\tilde{f}=\sum_{i=1}^q a_i 1_{C_i}$, for some sequence of real numbers $(a_i)_{i=1}^q$. By the definition of the topology of convergence on cylinders we know that $\lim_{m\to\infty}\int \tilde{f}d\mu_m=\int \tilde{f}d\mu$. Moreover, it follows from our construction that 
\begin{equation*}
\| f- \tilde{f} \|=\sup_{x\in\Sigma}|f(x)-\tilde{f}|\le \epsilon.
\end{equation*}
 The construction of $\tilde{f}$ can be made for every $\epsilon>0$. In particular we can construct a sequence $(\tilde{f}_k)_k$ such that $\| f-\tilde{f}_k \| \le 1/k$, and $\lim_{m\to \infty}\tilde{f}_kd\mu_m=\int \tilde{f}_kd\mu$. This immediately implies that $\lim_{m\to\infty}\int fd\mu_m=\int fd\mu,$ which completes the proof.
\end{proof}

\subsection{The space of test functions for the topology of convergence on cylinders}

The space of test functions for the weak* topology is $C_b(\Sigma)$. Similarly, the space of test functions for the vague topology is $C_c(\Sigma)$. For duality reasons it is actually convenient to have a Banach space as the space of test functions. For the vague topology this is not a serious issue, we can simply consider the closure of $C_c(\Sigma)$ in $C_b(\Sigma)$; this gives us the space $C_0(\Sigma)$ of functions that \emph{vanish at infinity}. More precisely, a function  $ f \in C_0(\Sigma)$ if it is a continuous functions such that for every $\epsilon >0$ there exists a compact set $K \subset \Sigma$ such that for every $x \in \Sigma \setminus K$ we have $|f(x)| < \epsilon$. 

It is a natural question to determine what is the space of test functions for the topology of convergence on cylinders. More precisely, determine a Banach space $V$ such that $(\mu_n)_n$ converges on cylinders to $\mu$ if and only if $\lim_{n \to \infty} \int f d\mu_n =
\int f d\mu$ for every $f \in V$.   From the definition of the topology of convergence on cylinders (Definition  \ref{def:topo.cyl}) it is clear that  the space $V$ should contain the characteristic function of a cylinder and finite linear combination of those. Define 
\begin{equation*}
H:=\left\{f\in C_b(\Sigma):f=\sum_{i=1}^n a_i1_{C_i}, \text{ where }a_i \in \R \text{ and  } C_i \text{ is a cylinder for each } i \right\}.
\end{equation*}
As with the vague topology, our space of test functions will be the closure of $H$ in $C_b(\Sigma)$, that we denote by $\bar{H}$. The following is direct from the definition of the topology of convergence on cylinders.

\begin{lemma}
Let $(\Sigma,\sigma)$ be a countable Markov shift and $\mu , (\mu_n)_n \subset \M_{\le 1}(\Sigma)$ then  $(\mu_n)_n$ converges on  cylinders to the measure $\mu$ if and only if for every $f \in \bar{H}$ we have
\begin{equation*}
\lim_{n \to \infty} \int f d \mu_n = \int f d \mu.
\end{equation*}
\end{lemma}

In what follows we will characterize the space $\bar{H}$, in order to do so we will require the following notions.

\begin{definition}
Let $(\Sigma, \sigma)$ be a countable Markov shift and $f:\Sigma \to \R$ a function. If $C$ is a cylinder of length $m$, denote by \begin{equation*}
 C(\ge n):= \left\{ x \in C : \sigma^m(x)\in \bigcup_{k\ge n}[k]  \right\}. 
\end{equation*} 
For a non-empty set $ A\subset \Sigma$ we define 
\begin{equation*}
var^A(f):=\sup  \left\{   \left|f(x)-f(y) \right| : (x,y) \in A\times A \right\}.
\end{equation*}
We declare $var^A(f)=0$ if $A$ is the empty set. 
\end{definition}

\begin{lemma} A function $f\in C_b(\Sigma)$ belongs to $\bar{H}$  if and only if the following three conditions hold:
\begin{enumerate}
\item \label{a} $f$ is uniformly continuous. 
\item \label{b} $\lim_{n\to\infty}\sup_{x\in [n]}|f(x)|=0.$
\item \label{c} $\lim_{n\to\infty}var^{ C(\ge n)}(f)=0,$   for every  cylinder $C\subset \Sigma$. 
\end{enumerate}
Moreover, if $\Sigma$ is locally compact, then $\bar{H}$ coincides with $C_0(\Sigma)$, the space of functions that vanish at infinity. 
\end{lemma}

\begin{proof}  Denote by $H_0$ the space of bounded functions satisfying conditions \eqref{a}, \eqref{b} and \eqref{c}.  The space $H_0$ is a closed subset of $C_b(\Sigma)$. The inclusion $H\subset H_0$ follows directly from the definition of $H$. Since $H_0$ is closed we obtain  that $\bar{H}\subset H_0$. 

We will now prove that $H_0\subset \bar{H}$. Fix $f\in H_0$ and  $\epsilon>0$. We will construct a function $g\in H$ such that $ \| f-g \|<\epsilon$. This would imply that $f \in \bar{H}$.

Since $f$ is uniformly continuous, there exists $q\in\N$ such that $var_{q}(f)<\epsilon$. Let $n_1\in\N$ be such that $\sup_{x\in [m]}|f(x)|<\epsilon,$ whenever $m> n_1$. Choose $n_2\in \N$ such that $var^{[i](\ge m)}(f)<\epsilon$, whenever $i\in \{1,...,n_1\}$ and $m> n_2$.  Similarly, choose $n_3\in \N$ such that $var^{[ij](\ge m)}(f)<\epsilon$, whenever $(i,j)\in\prod_{s=1}^2\{1,...,n_s\}$  and $m>n_3$. Inductively, we obtain a sequence $\{n_1,...,n_q,n_{q+1}\}$ such that for every $k\in \{1,...,q\}$ we have $var^{[i_1,...,i_k] (\ge m)}(f)<\epsilon$, whenever $(i_1,...,i_k)\in \prod_{s=1}^k\{1,...,n_s\}$ and $m>n_{k+1}$.

Let $f^* \in H_0$    and $C$ a non-empty cylinder of length $q$. We will define a number that depends on $f^*$ and $C$, which we denote by $l_{f^*}(C)$, as follows. Let us first assume that $C(\ge n)=\emptyset$, for some $n\in \N$. In this case we define $l_{f^*}(C)=0$. Now assume there exists a strictly increasing sequence $(n_k)_k\subset \N$ and points $x_k\in C\cap \sigma^{-q}[n_k]$. In this case we define $l_{f^*}(C):=\lim_{k\to\infty} f^*(x_k)$. It follows from  condition \eqref{c}  that $l_{f^*}(C)$ is well defined: it is independent of the sequences $(n_k)_k$ and $(x_k)_k$.

  A point $(a_1,...,a_k)\in  \prod_{s=1}^k\{1,...,n_s\}$ defines the  cylinder $[a_1,...,a_k]$. Collect all the non-empty cylinders that arise in this way and call this set $\Omega_k$. Define $\Omega=\bigcup_{k=1}^q\Omega_k$.  

Let us first prove the result when $l_f(C)=0$ holds for every $C\in \Omega$. By our choice of $\{n_1,...,n_{q+1}\}$ and the assumption $l_f(C)=0$, we know that for every $k\in \{1,...,q\}$ we have 
\begin{align}\label{+}\sup_{x\in [i_1,...,i_k](\ge m)}|f(x)|<\epsilon,\end{align}
 whenever $(i_1,...,i_k)\in \prod_{s=1}^k\{1,...,n_s\}$,  and $m>n_{k+1}$ (for consistency define the supremum over the empty set as zero). For every $C\in \Omega_q$ choose a point $x_C\in C$. Define $g:=\sum_{C\in \Omega_q}f(x_C)1_{C}$. Observe that if $x\in \bigcup_{C\in\Omega_q}C$, then $|f(x)-g(x)|<\epsilon$ (recall that $var_q(f)<\epsilon$). If $x$ does not belong to  $ \bigcup_{C\in\Omega_q}C$, then $x$ belongs to a cylinder of the form $[i_1,...,i_k,b]$, where $(i_1,...,i_k)\in  \prod_{s=1}^k\{1,...,n_s\}$ and $b>n_{k+1}$ (if $k=0$, then $x\in [b]$ where $b>n_1)$. By (\ref{+}) and our choice of $n_1$ we obtain that $|f(x)|< \epsilon.$ We therefore have $\|f-g \|<\epsilon$.

We will now explain how to reduce the general case to the situation where $l_f(C)=0$ holds for every $C\in \Omega$. 

Define $h_1:=\sum_{C\in \Omega_1}l_f(C)1_{C}$ and $f_1:=f-h_1$. We claim that $l_{f_1}(C)=0$, for every $C\in \Omega_1$. First suppose that there exists $n$ such that $C(\ge n)=\emptyset$. In this case, by definition, we have that $l_{f_1}(C)=0$. It remains to consider the non-trivial case in the definition of $l_{f_1}(C)$.  Let $(n_k)_k\subset \N$ be a strictly increasing sequence and $(x_k)_k$ points in $\Sigma$ such that $x_k\in C\cap \sigma^{-1}[n_k]$. Then
\begin{equation*}
l_{f_1}(C)=\lim_{k\to\infty}f_1(x_k)=\lim_{k\to\infty} \left(f(x_k)-h_1(x_k)\right) =l_f(C)-\lim_{k\to\infty}h_1(x_k).
\end{equation*}
Since $h_1=\sum_{C\in\Omega_1}l_f(C)1_{C},$ we know that $h_1(x_k)=l_f(C)$ (observe that all cylinders in $\Omega_1$ have length $1$). In particular $l_{f_1}(C)=0$. 

Now define $h_2:=\sum_{C\in \Omega_{2}}l_{f_1}(C)1_C$, and $f_2:=f_1-h_2$. We claim that $l_{f_2}(C)=0$, for every $C\in \Omega_{1}\cup \Omega_2$. As before, first suppose that there exists $n$ such that $C(\ge n)=\emptyset$. In this case, by definition, we have that $l_{f_2}(C)=0$. It remains to consider the non-trivial case in the definition of $l_{f_2}(C)$.  Let $C_1\in \Omega_1$ and $C_2\in \Omega_{2}$.  Choose a strictly increasing sequence $(n_k)_k\subset \N$ and  points $(x_k)_k$ in $\Sigma$ such that $x_k\in C_2\cap \sigma^{-2}[n_k]$. Then
\begin{equation*}
l_{f_2}(C_2)=\lim_{k\to\infty}f_2(x_k)=\lim_{k\to\infty} \left(f_1(x_k)-h_2(x_k) \right)=l_{f_1}(C_2)-\lim_{k\to\infty}h_2(x_k).
\end{equation*}
As before $h_2(x_k)=l_{f_1}(C_2)$, because $x_k\in C_2\cap \sigma^{-2}[k]$. We conclude that $l_{f_2}(C_2)=0$. Similarly, choose sequences $(m_k)_k\subset \N$ and $(y_k)_k\subset \Sigma$  such that $y_k\in C_1\cap \sigma^{-1}[m_k]$, then 
\begin{equation*}
l_{f_2}(C_1)=\lim_{k\to\infty}f_2(y_k)=\lim_{k\to\infty} \left( f_1(y_k)-h_2(y_k) \right)=l_{f_1}(C_1)-\lim_{k\to\infty}h_2(y_k).
\end{equation*}
Since $h_2$ is a finite linear combination of indicators of cylinders of length $2$ we obtain that $\lim_{k\to\infty}h_2(y_k)=0$. By construction we have that $l_{f_1}(C_1)=0$. We conclude that $l_{f_2}(C)=0$, for every $C\in \Omega_1\cup \Omega_2$. 

Continue this process and define $(h_k)_{k=1}^q$ and $(f_k)_{k=1}^q$ such that $f_{k}=f_{k-1}-h_{k}$, for every $k\in \{1,...,q\}$ (where we set $f_0=f$). By construction $l_{f_k}(C)=0$ for every cylinder $C\in\bigcup_{i=1}^k \Omega_i$. In particular $l_{f_q}(C)=0$, for every cylinder $C\in\Omega$. Finally, observe that $f_q=f-\sum_{k=1}^{q}h_k$. Since $h:=\sum_{k=1}^qh_k$ is a function in $H,$ it is enough to approximate $f_q$ (we can add back the function $h$ afterwards, $H$ is a vector space).

We will now assume that $\Sigma$ is locally compact. Since $\Sigma$ is locally compact the set $C(\ge m)$ is empty for large enough $m$. In particular condition \eqref{c} is always satisfied. By definition a function $f\in C_b(\Sigma)$ belongs to $C_0(\Sigma)$ if and only if condition \eqref{b} holds. In particular condition \eqref{b} implies condition \eqref{a}. Therefore $\bar{H}$ coincides with $C_0(\Sigma)$.

\end{proof}

By abuse of notation we denote by $C_0(\Sigma)$ the space of test functions for the topology of convergence on cylinders; this is reasonable because $\bar{H}=C_0(\Sigma)$ in the locally compact case. We say that $f$ \emph{vanishes at infinity} if $f\in C_0(\Sigma)$. To summarize, 
a function $f \in C_b(\Sigma)$ vanishes at infinity if it satisfies  conditions \eqref{a}, \eqref{b} and \eqref{c}.
%
%
It follows from the discussion above that the Banach space $C_0(\Sigma)$ is the space of test functions for the topology of convergence on cylinders. In particular,  the map $\mu\mapsto \int fd\mu$ is continuous in $\M_{\le1}(\Sigma,\sigma)$, whenever $f\in C_0(\Sigma)$.

\section{The space of invariant sub-probability measures is compact} \label{sec:sub}

We already noticed in  Lemma \ref{lem:comp} that if $(\Sigma, \sigma)$ is a locally compact transitive countable 
Markov shift, the space of invariant sub-probability measures  $\M_{\le 1}(\Sigma,\sigma)$ is compact with respect to the vague topology. It is a consequence of Lemma \ref{lem_loc_cyl} that  $\M_{\le 1}(\Sigma, \sigma)$ is also compact with respect to the topology of convergence on cylinders. In this section we prove that  the space $\M_{\le 1}(\Sigma, \sigma)$ is compact with respect to the topology of convergence on cylinders for a larger class of transitive  countable Markov shifts, that is, for countable Markov shifts with the $\F-$property (see Definition \ref{def:F}). Our results are also sharp: if $(\Sigma,\sigma)$ does not satisfy the $\F-$property, then there are sequences of periodic measures that converge to a finitely additive measure that is not countably additive (see Proposition \ref{prop:noF}). Our next result states that invariance is preserved under limits provided the limiting object is a countably additive measure. 

\begin{lemma}\label{invariance} Let $(\Sigma, \sigma)$ be a countable Markov shift. Let $(\mu_n)_n$ be a sequence of invariant probability measures converging on cylinders to a sub-probability measure $\mu$. Then $\mu$ is an invariant measure. 
\end{lemma}

\begin{proof}  The measure $\mu$ is invariant if it is equal to $\sigma_*\mu:= \mu(\sigma^{-1})$. Note that, in order to prove the invariance, it is enough to prove that $\sigma_*\mu(D)=\mu(D)$, for every cylinder $D$. Observe that if $D$ is a cylinder then $\sigma^{-1}D=\bigcup_{i\ge 1}D_i$, where $(D_i)_i$ is a finite or countable collection of cylinders. Since $\mu_n$ is invariant we have that $\mu_n(D)=\mu_n(\sigma^{-1}D)$. If $\sigma^{-1}D=\bigcup_{i=1}^m D_i$ is a finite union of cylinders we obtain that  
\begin{equation*}
\mu(D)= \lim_{n\to\infty}\mu_n(D)=\lim_{n\to\infty}\mu_n(\sigma^{-1}D)= \lim_{n\to\infty}\mu_n\left(\bigcup_{i=1}^m D_i \right)= \mu(\sigma^{-1}D).
\end{equation*}
If $\sigma^{-1}D$ is union of infinitely many cylinders we have 
\begin{equation*}
\mu(D)=\lim_{n\to\infty}\mu_n(D)=\lim_{n\to\infty}\mu_n(\sigma^{-1}D)\ge \lim_{n\to\infty}\mu_n\left(\bigcup_{i=1}^M D_i\right)=\mu\left(\bigcup_{i=1}^M D_i\right),
\end{equation*}
for every $M \in \N$. We conclude that $\mu(D)\ge \mu(\sigma^{-1}D)$. We therefore proved that for every cylinder $D$ we have $\mu(D)\ge \mu(\sigma^{-1}D)$. Suppose $D$ is a cylinder of length $s$   and enumerate all cylinders of length $s$ by $(E_k)_k$ with $E_1=D$.
Since $\mu(E_k)\ge \mu(\sigma^{-1}E_k)$, for every $k\in \N$, we obtain
\begin{align}\label{1} 
 \mu\left(\bigcup_{k\ge1} E_k \right)\ge \mu\left(\bigcup_{k\ge1}\sigma^{-1}E_k\right).\end{align}  
Observe that $(E_k)_k$ and $(\sigma^{-1}E_k)_k$ are partitions of $\Sigma$, in particular $\Sigma=\bigcup_{k\ge1} E_k=\bigcup_{k\ge1}\sigma^{-1}E_k$. This implies that (\ref{1}) is an equality, therefore $\mu(E_k)=\mu(\sigma^{-1}E_k)$, for every $k\in \N$. In particular we obtained that $\mu(D)=\mu(\sigma^{-1}D)$, as desired. 
\end{proof}

\begin{remark}\label{rem:dense} Recall that if $(\Sigma, \sigma)$ is transitive then the periodic measures are dense in $\M(\Sigma,\sigma)$ with respect to the weak* topology (see Theorem \ref{lem:cs}). It is a consequence of Lemma \ref{equivtop} that the same holds for the topology of convergence on cylinders. In other words, given an invariant probability measure $\mu$, there exists a sequence $(\mu_n)_n$ of periodic measures such that $\lim_{n\to\infty}d(\mu,\mu_n)=0$.
\end{remark}

%

\begin{remark}
Recall that every element of $\M_{\le 1}(\Sigma, \sigma)$ is of the form $\lambda \mu$, where $\mu \in \M(\Sigma, \sigma)$ and
$\lambda \in [0,1]$.
\end{remark}

In order to study the space  $\M_{\le 1}(\Sigma, \sigma)$  we will model it with  a space of functions.

\begin{definition} \label{def:L} Denote by $FinCyl(\Sigma)$ the collection of non-empty finite unions of cylinders in $\Sigma$. Let $M(\Sigma)$ be the space of functions $F:FinCyl(\Sigma)\to [0,1]$ and $L(\Sigma) \subset M(\Sigma)$ the space of functions  satisfying the following conditions.
\begin{enumerate}
\item \label{La} If $C\subset C'$ are cylinders, then $F(C)\le F(C')$.  
\item \label{Lb} $F(\bigcup_{k=1}^n C_k)=\sum_{k=1}^n F(C_k)$, for every (finite) collection of disjoint cylinders $(C_k)_k$.
\end{enumerate}
\end{definition}

\begin{remark}
Observe that  $FinCyl(\Sigma)$ is a countable set. Fix a bijection between $FinCyl(\Sigma)$ and $\N$. This bijection allow us to identify   $M(\Sigma)$ with $[0,1]^\N$.  From now on we consider  $L(\Sigma)$ as a subset of $[0,1]^\N$. We endow $M(\Sigma)$ with the product topology. Observe that $(R_n)_n\subset M(\Sigma)$ converges to $R\in M(\Sigma)$ if and only if 
\begin{equation*}
\lim_{n\to\infty}R_n(D)=R(D),
\end{equation*}
 for every $D\in FinCyl(\Sigma)$. 
\end{remark}

\begin{remark} \label{r:c}
Observe that the metric $d$,   see equation (\ref{metric}), also defines a metric on $L(\Sigma)$. Indeed, if $F,G \in L(\Sigma)$ then
\begin{align} 
d(F,G)=\sum_{n\ge1} \frac{1}{2^n}|F(C_n)-G(C_n)|,
\end{align}
where $(C_n)_n$ is some enumeration of the cylinders on $\Sigma$, is a metric on $L(\Sigma)$. It is important to observe that the topology induced by $d$ on $L(\Sigma)$ is compatible with the product topology on $M(\Sigma)=[0,1]^\N$. Indeed, $\lim_{n\to\infty}d(F_n,F)=0$, if and  only if $\lim_{n\to\infty}F_n(C)=F(C)$, for every cylinder $C$. By condition \eqref{Lb}  in the definition of $L(\Sigma)$ this is equivalent to $\lim_{n\to\infty}F_n(D)=F(D)$, for every $D\in FinCyl(\Sigma)$. In other words, there exists a continuous injective map from the set $L(\Sigma)$, endowed with the topology generated by the metric $d$, into    the space $[0,1]^\N$ endowed with the product topology.
\end{remark} 
 
\begin{lemma} 
The set $L(\Sigma)$ is compact  with respect to the topology induced by $d$. 
\end{lemma} 
 
\begin{proof} 
We will first prove that $L(\Sigma)$ is a closed subset of $M(\Sigma)$. Let $(F_n)_n$ be a sequence of functions in $L(\Sigma)$ that converges to $F\in M(\Sigma)$.  Let $C$ and $D$ be cylinders such that $C\subset D$. Then $F_n(C)\le F_n(D)$, for every $n\in \N$. We conclude that 
\begin{equation*}
F(C)=\lim_{n\to \infty}F_n(C)\le \lim_{n\to\infty}F_n(D)=F(D).
\end{equation*}
Similarly, if  $(C_k)^m_{k=1}$ is a finite collection of disjoint cylinders we have that 
\begin{equation*}
F \left(\bigcup_{k=1}^m C_k \right)=\lim_{n\to\infty} F_n \left(\bigcup_{k=1}^m C_k \right)=\lim_{n\to\infty} \sum_{k=1}^m F_n(C_k)=\sum_{k=1}^m F(C_k).
\end{equation*}
We conclude that $F\in L(\Sigma)$. It follows that $L(\Sigma)$ is a closed subset of $M(\Sigma)$. Since $[0,1]^\N$ is compact, by virtue of Remark \ref{r:c} we have that  $L(\Sigma)$ is compact with the topology induced by $d$. 
\end{proof}

\begin{remark}\label{embedding} Observe that every sub-probability measure on $\Sigma$ can be identified with a unique function $F\in L(\Sigma)$. More precisely, given $\mu\in \M_{\le1}(\Sigma,\sigma)$ we define $F_\mu\in L(\Sigma)$ by $F_\mu(D):=\mu(D)$, for every $D\in FinCyl(\Sigma)$. The map $\mu\mapsto F_\mu$ defines a continuous embedding $\M(\Sigma,\sigma) \hookrightarrow L(\Sigma)$, when we endow $\M(\Sigma,\sigma)$ with the topology of convergence on cylinders. We say that a sequence $(\mu_n)_n\subset \M(\Sigma,\sigma)$ converges to $F\in L(\Sigma)$ if $(F_{\mu_n})_n\subset L(\Sigma)$ converges to $F$. \end{remark}

In light of Remark \ref{embedding}, in order to prove that the space of invariant  sub-probability measures is compact with respect to the cylinder topology it suffices to prove that  $\overline{\M(\Sigma,\sigma)}\subset L(\Sigma)$ consist of invariant sub-probability measures. At this point we will make a further assumption on the countable Markov shifts considered.

\begin{definition} \label{def:F}
A countable Markov shift $(\Sigma, \sigma)$ is said to satisfy the $\F-$\emph{property} if  for every element of the alphabet ${\bf i}$ and natural number $n$, there are only finitely many admissible words of length $n$ starting and ending at ${\bf i}$.
\end{definition}

\begin{remark}
Every countable Markov shift $(\Sigma, \sigma)$ of finite topological entropy and every locally compact countable Markov shift satisfies the $\F-$property.  There also exists infinite entropy non-locally compact countable Markov shifts satisfying  the $\F-$\emph{property}. Indeed, let $(a_n)_n$ be a sequence of positive integers such that 
\begin{equation*}
\lim_{n \to \infty} \frac{1}{n} \log a_n = \infty.
\end{equation*} 
Consider the countable Markov shift defined by a graph made of $a_n$ simple loops of length $n$ which are based at a common vertex and otherwise do not intersect. This system has the desired properties.
\end{remark}

\begin{proposition}\label{extension} Let $(\Sigma,\sigma)$ be a transitive countable Markov shift satisfying the $\F-$property. If $(\mu_n)_n$ is a sequence of periodic measures converging to a function $F\in L(\Sigma)$, then $F$ extends to an invariant sub-probability measure. 
\end{proposition}

\begin{proof}  We start by proving that $F$ extends to a measure. Fix a cylinder $C=[a_1,...,a_m]$, and denote by $Ck$ the cylinder $[a_1,...,a_m,k]$.  We will need the following lemma.

\begin{lemma}\label{lem:mea}  \begin{align}\label{goal} F(C)=\sum_{k\ge1}F(Ck).\end{align}
\end{lemma}
We assume that $F(C)>0$, otherwise there is nothing to prove (both left and right hand side would be zero). From now on assume that $n$ is sufficiently large so that  $\mu_n(C)>0$. Let $p_n$ be a periodic point associated to $\mu_n$ such that $p_n\in C$.

\begin{proof}[Proof of Lemma \ref{lem:mea}] Observe that 
\begin{equation*}
F(C)-\sum_{s=1}^{k-1}F(Cs)=\lim_{n\to\infty}\left(\mu_n(C)-\sum_{s=1}^{k-1}\mu_n(Cs)\right)=\lim_{n\to\infty}\mu_n \left(\bigcup_{s\ge k}Cs \right),
\end{equation*}
therefore Lemma \ref{lem:mea} is equivalent to prove that $\lim_{k\to\infty} \lim_{n\to\infty}\mu_n(\bigcup_{s\ge k} Cs)=0$. 
We will argue by contradiction and assume that  
\begin{equation*}
\lim_{k\to\infty} \lim_{n\to\infty}\mu_n\left(\bigcup_{s\ge k} Cs \right)=A>0.
\end{equation*}
Observe that  $\big(\lim_{n\to\infty}\mu_n(\bigcup_{s\ge k} Cs)\big)_k$ decreases as $k$ goes to infinity. We obtain that $\lim_{n\to\infty}\mu_n(\bigcup_{s\ge k} Cs)\ge A$, for every $k\in \N$. 

Recall that $C=[a_1,...,a_m]$ and define the set  $Q\subset \N$ by the following rule: $q\in Q$ if and only if ${\bf a_mq}$ is an admissible word. Define a function $p:Q\to\Z$ as follows: $p(i)=k$ if there exists an admissible word starting at $i$ and ending at $a_1$ of length $k+1$, but there is not any such word of length  less or equal to $ k$. The map $p$ is proper, in other words, $p^{-1}([a,b])$ is finite for every $a,b\in \R$.  Indeed, assume by contradiction that $p^{-1}([a,b])$ is infinite, this would imply that $p^{-1}(c)$ is infinite for some $c\in \N$. For each $w\in p^{-1}(c)$ we have an admissible word of length $c+1$ connecting $w$ and $a_1$, this will create an admissible word (with length $m+c+1$) of the form ${\bf a_1...a_mw...a_1}$. This contradicts the fact that $(\Sigma, \sigma)$ satisfies the $\F-$property. We conclude that $p:Q\to\Z$ is proper. 

Choose $k_0\in \N$ such that $p(s)\ge \lfloor\frac{4}{A}\rfloor+1$, for every $s\in Q$ satisfying $s\ge k_0$. Recall that $p_n$ is a periodic point associated to the measure $\mu_n$. We denote the minimal period of $p_n$ by $m+t_n$, and let $[a_1...a_mb_1...b_{t_n}]$ be a neighborhood  of $p_n$. By the definition of $\mu_n$ we know that $\mu_n(Cs)$ is approximately the number of times that the word ${\bf a_1...a_ms}$ appears in ${\bf w_n:= a_1...a_mb_1...b_{t_n}a_1...a_m}$, divided by $m+t_n$. If $s\ge k_0$, then each block ${\bf a_1...a_ms}$ appearing in  ${\bf w_n}$ is contained in a longer block of the form ${\bf a_1...a_msr_1...r_B}$, where $B\ge \lfloor\frac{4}{A}\rfloor+1$, and $r_t\ne a_1$, for all $t\in \{1,...,B\}$. In particular, for $s\ge k_0$, each block of the form ${\bf a_1...a_ms}$  generates $B$ letters that do not contribute to the number of blocks ${\bf a_1...a_m}$ in ${\bf w_n}$. Choose $n_0$ such that $\mu_n(\bigcup_{s\ge k_0} Cs)\ge \frac{A}{2}$, for every $n\ge n_0$.  This implies that the number of blocks of the form ${\bf a_1...a_ms}$, where $s\ge k_0$, in ${\bf w_n}$ is at least $(m+t_n)A/2$. As explained above, each of those blocks  generate a  disjoint block of length $(m+B+1)$. The number of letters used in those disjoint blocks add up to $(m+B+1)(m+t_n)A/2$. Observe that $(m+B+1)A/2>1$, which contradicts that the total number of letters is $(m+t_n)$. We conclude that $A=0$.
\end{proof}

We will now use Kolmogorov's extension theorem to prove that $F$ comes from a measure on $\Sigma$. To each $I_n:=\{1,...,n\}\subset \N$ we associate a measure on $\N^{I_n}$: this is the atomic measure $\nu_n$ that assigns to $[m_1,...,m_n]$ the number  $F([m_1,...,m_n])$. We remark that if ${\bf m_1...m_n}$ is not an admissible word of $\Sigma$, then $F([m_1,...,m_n])=0$. In order to use Kolmogorov's extension theorem and obtain a measure on $\N^\N$ we need to verify the consistency of the family $(\nu_n)_n$, in other words, that $$\nu_n((m_1,...,m_n))=\nu_{n+1}((m_1,...,m_n)\times \N).$$  By definition of the family $(\nu_n)_n$ this is equivalent to the formula $$F(D)=\sum_{k\ge 1}F(Dk),$$ for $D=[m_1,...,m_n]$. Lemma \ref{lem:mea} implies the consistency of $(\nu_n)_n$. It  follows from Kolmogorov's extension theorem that $F$ extends to a measure $\mu$ on the full shift $\N^\N$. Observe that by definition of $F$ the measure $\mu$ is supported on $\Sigma\subset \N^\N$.  The invariance of $\mu$ follows from Lemma  \ref{invariance}.  
\end{proof}


\begin{remark}[Limits are not always measures] \label{nomealim}We now exhibit examples of countable Markov shifts that do not satisfy the $\F-$property for which sequences of measures converge to a function $F$ that can not be extended to a measure. Let $\Sigma=\N^\N$ be the full shift. Consider the periodic point $p_n=\overline{1n}$, and denote by $\mu_n$ the periodic measure associated to $p_n$. Observe that $(\mu_n)_n$ converges to $F\in L(\Sigma),$ where $F$ is given by $F([1])=1/2$, and $F(C)=0$, for any other cylinder $C\subsetneq\Sigma$. In this case it is clear that $F$ does not come from a measure: use the decomposition $[1]=\bigcup_{s\ge1}[1s]$, and the definition of $F$. Equivalently, the formula $F([1])=\sum_{s\ge1}F([1s])$, does not hold. In the full shift we can not expect to always have a measure as the limit  of probability measures (in the topology of convergence on cylinders). Similar examples are easy to construct. For instance consider the countable Markov shift defined by the matrix $M=(M_{ij})$, where $M_{1k}=1=M_{k1}$, for all $k\in \N$, and $M_{ij}=0$ for the remaining entries. In this case the same choice of measures $(\mu_n)_n$ would provide a sequence of invariant probability measures that do not converge to a countably additive measure.
\end{remark}

\begin{proposition}\label{compactness} Let $(\Sigma,\sigma)$ be a transitive countable Markov shift satisfying the $\F-$property. Then any sequence of invariant probability measures $(\mu_n)_n$ has a subsequence  that converges on cylinders to an invariant sub-probability measure. 
\end{proposition}

\begin{proof}
Since  $\M(\Sigma,\sigma)\subset L(\Sigma)$, by compactness of $L(\Sigma)$ there exists a subsequence $(\mu_{n_k})_k$ converging to a function $F\in L(\Sigma)$. Since the periodic measures are dense in $\M(\Sigma,\sigma)$ (see Remark \ref{rem:dense}) we can find a sequence of periodic measures $(\nu_k)_k$ such that $d(\mu_{n_k},\nu_k)\le \frac{1}{k}$. It follows that $\lim_{k\to\infty}d(\nu_k,F)=0$. We can now use Proposition \ref{extension} and conclude that $F$ corresponds to an invariant  sub-probability measure.
\end{proof}

\begin{remark}\label{rem:compactness} The proof of Proposition \ref{compactness} also implies that $\overline{\M(\Sigma,\sigma)}\subset \M_{\le 1}(\Sigma,\sigma)$. Indeed, if $F\in \overline{\M(\Sigma,\sigma)}$, then we have a sequence of invariant probability measures $(\mu_n)_n$ converging to $F$. As in the proof of Proposition \ref{compactness} we conclude that $F$ can be approximated by periodic measures, and therefore Proposition \ref{extension} implies the result. 
\end{remark}

As mentioned in the introduction, to prove that $\M_{\le1}(\Sigma,\sigma)$ is affine homeomorphic to the Poulsen simplex we need to prove the existence of a sequence of invariant measures that converges on cylinders to the zero measure. In our next result  we obtain such property. We emphasize that if $(\Sigma,\sigma)$ does not satisfy the $\F-$property, then this is not necessarily true (see Example \ref{ex:nozero}).

\begin{lemma}\label{tozero} Let $(\Sigma,\sigma)$ be a transitive countable Markov shift satisfying the $\F-$property. Then there exists a sequence of invariant probability measures converging on cylinders to the zero measure. 
\end{lemma}

\begin{proof} Fix some natural number $k$. We say that Property $(k)$ holds if there exist arbitrarily long admissible words of the form ${\bf a_1...a_m}$, where $\{a_1, a_m\}\subset \{1,...,k\}$, and $a_i\ge k+1$, for all $i\in \{2,...,m-1\}$. If Property $(k)$ holds we can construct a sequence of periodic measures $(\mu_n^{(k)})_n$ such that $\lim_{n\to\infty}\mu_n^{(k)}(\bigcup_{s=1}^k [s])=0$. First observe that there exists $M_0=M_0(k)$ such that every two letters in $\{1,...,k\}$ can be connected with an admissible word of length less or equal to   $M_0$. By hypothesis for every $n\in\N$ there exists an admissible word ${\bf w_n= a^{(n)}_1...a^{(n)}_{m_n}}$, where $\{a^{(n)}_1, a^{(n)}_{m_n}\}\subset \{1,...,k\}$, and $a^{(n)}_i\ge k+1$, for all $i\in \{2,...,m_n-1\}$, and $m_n\ge n$. We can extend the word ${\bf w_n}$ into an admissible word ${\bf w_n'=a_1^{(n)}...a^{(n)}_{m_n}b^{(n)}_1...b^{(n)}_{s_n}a^{(n)}_1}$, where $s_n\le M_0$. The word ${\bf w_n'}$ can be used to define a periodic orbit, and therefore a periodic measure, say $\mu_n^{(k)}$, on $\Sigma$. Observe that 
\begin{equation*}
\mu_n^{(k)}\left(\bigcup_{s=1}^k[s]\right)\le \dfrac{s_n+2}{s_n+m_n}\le \dfrac{M_0+2}{n}, 
\end{equation*}
which readily implies that $\lim_{n\to\infty}\mu_n^{(k)}(\bigcup_{s=1}^k [s])=0$.

We will now verify that under the hypothesis of  Lemma \ref{tozero} Property $(k)$ holds.  Assume by contradiction that this is not possible, in other words that any such word has length less or equal to   $N_0$. Define 
\begin{equation*}
T:\{n\in\N: n\ge k+1\}\to \N,
\end{equation*}
in the following way: $T(n)=r$, if there exists an admissible word of length $r$ with first letter in $\{1,...,k\}$ and ending at $n$, but there is no such admissible word of length strictly  less to   $r$. Similarly define 
\begin{equation*}
S:\{n\in\N: n\ge k+1\}\to \N,
\end{equation*}
in the following way: $S(n)=r$, if there exists an admissible word of length $r+1$ with first letter $n$ and ending at some letter in $\{1,...,k\}$, but there is not such admissible word of length  less or equal to $r$. By definition of $T$ and $S$ we know that given $n\ge k+1$, there exists an admissible word  ${\bf y_n:= c^{(n)}_1...c^{(n)}_{T(n)-1}nd^{(n)}_1...d^{(n)}_{S(n)}}$, where $\{c^{(n)}_1,d^{(n)}_{S(n)}\}\subset \{1,...,k\}$ and the rest of the letters are strictly larger than $k$. Observe that by assumption we have $T(n)+S(n)\le N_0$, for every $n\ge k+1$. For $n\ge k+1$ define $W(n)$ as the biggest letter in the word ${\bf y_n}$. We can  inductively choose a sequence $(n_t)_t$ such that $W(n_t)<n_{t+1}$, and observe that $({\bf y_{n_t}})_t$ are pairwise distinct. As with the words $({\bf w_n})_n$,  we can extend each ${\bf y_{n}}$ to an admissible word  ${\bf y_{n}':= e^{(n)}_1...e^{(n)}_{s_n}y_n f^{(n)}_1...f^{(n)}_{r_n} }$, where $s_n$ and $r_n$ are less than $ M_0$, and  $ e^{(n)}_1=1= f^{(n)}_{r_n}$. The word ${\bf y_n'}$ defines a periodic point of period $\le 2M_0+N_0$. Since $({\bf y_{n_t}})_t$ are pairwise distinct we found infinitely many periodic points of periods 
less or equal to   $2M_0+N_0$ (starting and ending at $1$), which contradicts that $(\Sigma,\sigma)$ satisfies the $\F-$property. We conclude that Property $(k)$ holds for every $k\in \N$. 

For every $k\in \N$ we obtain a sequence of periodic measures $(\mu_n^{(k)})_n$ such that $\lim_{n\to\infty}\mu_n^{(k)}(\bigcup_{s=1}^k[s])=0$. Let $n_k$ be such that $\mu_{n_k}^{(k)}(\bigcup_{s=1}^k[s])\le \frac{1}{k}$. To simplify notation we define $\nu_k:=\mu_{n_k}^{(k)}$. We claim that $(\nu_k)_k$ converges on cylinders to the zero measure. Observe that for $k\ge m$, we have 
\begin{equation*}
\nu_k([m])\le \nu_k\left(\bigcup_{s=1}^m[s] \right)\le\nu_k \left(\bigcup_{s=1}^k[s] \right)\le \frac{1}{k}.
\end{equation*}
We conclude that $\lim_{k\to\infty}\nu_k([m])=0$. Since $m\in \N$ was arbitrary, and every cylinder $C$ is contained in a cylinder of length one,  we conclude that $(\nu_k)_k$ converges on cylinders to the zero measure. 
\end{proof}

\begin{example}\label{ex:nozero}
We exhibit an example of a countable Markov shift of infinite entropy  not satisfying the $\F-$property, for which there is no sequence of measures converging to zero in the cylinder topology. Let $(\Sigma, \sigma)$ be the countable Markov shift defined by the graph formed by infinitely many loops of length two rooted at a common vertex. That is, the allowed transitions are of the form $1 \rightarrow N$ and  $N \rightarrow 1$ for every $N \in \N$, this example was also considered in Remark \ref{nomealim}. The system has infinite entropy, since it has infinitely many periodic orbits of period two intersecting $[1]$. The frequency of the digit $1$ is at least $1/2$ for every element of $\Sigma$. Therefore, if $\mu$ is an ergodic measure then $\mu([1])\geq 1/2$. Thus, for any sequence of invariant measures $(\mu_n)_n$ we must have that $\liminf_{n \to \infty} \mu_n([1]) \geq 1/2$. In particular, the sequence $(\mu_n)_n$ does not converges  to zero in the cylinder topology.
\end{example}

We can now prove the compactness of the space of sub-probability measures. As explained in the introduction a compactification of the space of invariant probability measures is important for applications (for instance see \cite{itv} and \cite{v}).

\begin{theorem} \label{sub:comp}
If $(\Sigma,\sigma)$ is a transitive countable Markov shift satisfying the $\F-$property. Then the space  $\M_{\le 1}(\Sigma, \sigma)$ is compact with respect to the  topology of convergence on cylinders.
\end{theorem}

\begin{proof}
 It is a consequence of  Remark \ref{rem:compactness} that $\overline{\M(\Sigma,\sigma)}\subset \M_{\le1}(\Sigma,\sigma)$. It is enough to prove that $\overline{\M(\Sigma,\sigma)}=\M_{\le1}(\Sigma,\sigma)$. Let $(\mu_n)_n$ be a sequence of invariant probability measures converging on cylinders to the zero measure (see Lemma \ref{tozero}). An element in $\M_{\le1}(\Sigma,\sigma)$ has the form $\lambda \mu$, where $\mu$ is an invariant probability measure and $\lambda\in [0,1]$. Define $\nu_n=\lambda\mu+(1-\lambda)\mu_n$. Observe that $(\nu_n)_n$ convergences on cyliders to $\lambda \mu$. This concludes that $\overline{\M(\Sigma,\sigma)}=\M_{\le1}(\Sigma,\sigma)$, and therefore $\M_{\le1}(\Sigma,\sigma)$ is compact. 
\end{proof}

The idea behind Remark \ref{nomealim} can be used to prove that Theorem \ref{sub:comp} is sharp. We will prove that without the $\F-$property it is possible to construct a sequence of invariant measures that converges on cylinders to a finitely additive measure that is not countably additive. In particular, Theorem \ref{sub:comp} is false without the $\F-$property assumption.

\begin{proposition}\label{prop:noF} Suppose that $(\Sigma,\sigma)$ does not satisfy the $\F-$property. Then there exists a sequence of periodic measures that converges on cylinders to $F\in L(\Sigma)$, where $F$ can not be extended to a measure. 
\end{proposition}

\begin{proof} Since $(\Sigma,\sigma)$ does not satisfy the $\F-$property there exists a symbol ${\bf i}$ and natural number $n$ such that there are infinitely many admissible words of length $n$ that start and end at ${\bf i}$. The set of admissible words of length $k+1$ starting and ending at ${\bf i}$, where the symbol ${\bf i}$ only appears at the beginning and at the end of the word is denoted by $A_k$.  By hypothesis there exists $q\le n$ such that $|A_q|=\infty$. Set $A_q=\{w_k:k\in \N\}$. Observe that each $w_k\in A_q$ defines a periodic measure that we denote by $\mu_k$. Maybe after passing to a subsequence we can assume that $(\mu_k)_k$ converges on cylinders to $F\in L(\Sigma)$. By construction we know that $\mu_k([{\bf i}])=\frac{1}{q}$. Observe that $\mu_k([{\bf i r}])$ is equal to $0$ or $\frac{1}{q}$, for every $k$ and ${\bf r}$. If $\lim_{k\to\infty}\mu_k([{\bf i r}])=0$, for every ${\bf r}$, then $F$ can not come from a measure: $\sum_{r\ge 1}F([{\bf i r}])=0$, but $F([{\bf i}])=\frac{1}{q}$. We assume there exists ${\bf r_1}$ such that $\lim_{k\to\infty}\mu_k([{\bf i r_1}])=\frac{1}{q}$, which is equivalent to say that $\mu_k([{\bf i r_1}])=\frac{1}{q}$, for every $k$ sufficiently large. We can repeat the process and conclude that if $\lim_{k\to\infty}\mu_k([{\bf i r_1 s}])=0$, for every ${\bf s}$, then $F$ does not come from a measure. We can assume that there exists ${\bf r_2}$ such that $\mu_k([{\bf ir_1 r_2}])=\frac{1}{q}$, for $k$ sufficiently large. By repeating this process we obtain that $F$ does not come from a measure or that $\mu_k([{\bf i r_1 ... r_{q-1}}])=\frac{1}{q}$, for $k$ sufficiently large. This last condition is equivalent to say that the sequence $(\mu_k)_k$ stabilizes, which contradicts that the measures are pairwise different.

\end{proof}

In Section \ref{sec:sus} we will be interested in countable Markov shifts that do not necessarily have the $\F-$property. Despite  Proposition \ref{prop:noF} we can regain control by imposing an integrability condition on the sequence of probability measures (see Proposition \ref{prop:F}). This integrability condition will rule out the sequence constructed in Proposition \ref{prop:noF}. 


\section{The Poulsen simplex} \label{sec:pou}

We now prove one of our main results,  in which we characterize the spaces $\M_{\leq 1}(\Sigma, \sigma)$ and $\M(\Sigma, \sigma)$ for  countable Markov shifts satisfying the $\F-$property.

\begin{theorem} \label{thm:po_main}
Let $(\Sigma, \sigma)$ be a transitive countable Markov shift satisfying the $\F-$property. Then  $\M_{\le 1}(\Sigma,\sigma)$ is affine homeomorphic to the Poulsen simplex.
\end{theorem}

\begin{proof}
An element in $\M_{\le1}(\Sigma,\sigma)$ has the form $\lambda \mu$, where $\mu$ is an invariant probability measure and $\lambda\in [0,1]$.  To prove that $\M_{\le1}(\Sigma,\sigma)$ is the Poulsen simplex it is enough to prove that the extreme points of $\M_{\le1}(\Sigma,\sigma)$ are dense (we already know that $\M_{\le 1}(\Sigma,\sigma)$ is a metrizable convex compact Choquet simplex). We will approximate the measure $\lambda\mu$ with periodic measures. As explained in the proof of Theorem \ref{sub:comp}, we can construct invariant probability measures $(\nu_n)_n$ such that $\lim_{n\to\infty}d(\nu_n,\lambda \mu)=0$. By Remark \ref{rem:dense} we can find a sequence $(\overline{\nu}_n)_n$ of periodic measures such that $d(\overline{\nu}_n,\nu_n)\le \frac{1}{n}$. This implies that $(\overline{\nu}_n)_n$ converges on cylinders to $\lambda\mu$, and therefore, by the main result of \cite{los},  $\M_{\le 1}(\Sigma,\sigma)$ is affine homeomorphic to the Poulsen simplex.
\end{proof}

Since the extreme points of $\M_{\le 1}(\Sigma,\sigma)$ are the ergodic probability measures together with the zero measures, it  follows directly from Theorem \ref{thm:po_main} and Lemma \ref{equivtop} that,

\begin{theorem}
Let $(\Sigma, \sigma)$ be a transitive countable Markov shift satisfying the $\F-$property. Then $\M(\Sigma, \sigma)$ is affinely homeomorphic to the Poulsen simplex minus a vertex and all of its convex combinations.
\end{theorem}

\begin{corollary}
Let $(\Sigma, \sigma)$ be a transitive countable Markov shift satisfying the $\F-$property. Then
the set $\M_e(\Sigma,\sigma)$ is path connected.
\end{corollary}

\begin{proof}
The set of extreme points of the Poulsen simplex is path connected \cite[(4) p.101]{los}. It follows from Theorem \ref{thm:po_main}  that the set $\M_e(\Sigma,\sigma)\cup \{0_m\}$, where $0_m$ denotes the zero measure is path connected. Denote by $Q:= [-1,1]^{\N}$ the Hilbert cube and let  $P:= \{ (x_1, x_2, \dots) \in Q : |x_n|<1 , \text{ for every } n \in \N \}$.
It was proved in \cite[Theorem 3.1]{los} that there exists a homeomorphism $h$ between the Hilbert cube $Q$ and the Poulsen simplex   
which maps $P$ onto the set of extreme points of the Poulsen simplex.  Denote by $z=h^{-1}(0_m) \in P$. For any $x,y \in P \setminus \{z\}$ it is clear that there exists
a continuous path $p:[0,1] \to P$ such that $p(0)=x, p(1)=y$ and $p(t) \neq z$ for every $t \in [0,1]$. Therefore the set $\M_e(\Sigma,\sigma)$ is path connected.
\end{proof}

\section{The space of invariant measures for suspension flows}\label{sec:sus}
In this section we study the space of invariant probability measures of a suspension flow defined over a countable Markov shift.

\subsection{Suspension flows}
Let $(\Sigma, \sigma)$ be a countable Markov shift and $\tau:\Sigma \to \R$ a continuous positive function bounded away from zero, that is, there exists $c=c(\tau)>0$ such that $\tau(x)\ge c$, for all $x\in \Sigma$.  Consider the space
\begin{equation*} 
Y= \left\{ (x,t)\in \Sigma  \times \R \colon 0 \le t \le\tau(x) \right\},
\end{equation*}
with the points $(x,\tau(x))$ and $(\sigma(x),0)$ identified for
each $x\in \Sigma $. The \emph{suspension flow} over $\sigma$
with \emph{roof function} $\tau$ is the semi-flow $\Theta = (
\theta_t)_{t \ge 0}$ on $Y$ defined by
\begin{equation*}
 \theta_t(x,s)= (x,s+t), \ \text{whenever $s+t\in[0,\tau(x)]$.}
\end{equation*}
In particular, $\theta_{\tau(x)}(x,0)= (\sigma(x),0)$. The space of invariant probability measures for the shift is related to the space of invariant probability measures for the flow, that we denote by $\M(\Sigma,\sigma,\tau)$. Indeed, it follows from a classical result of Ambrose and Kakutani \cite{ak} (see \cite[Chapter 6]{pp} for details) that,

\begin{lemma}  \label{lem:medi}
Let $(Y, \Theta)$ be a suspension flow over $(\Sigma, \sigma)$ with roof function $\tau$ bounded away from zero. Let
\begin{equation*}
\M_\tau:=\left\{\mu\in \M(\Sigma,\sigma):\int \tau d\mu<\infty \right\}.
\end{equation*}
The map $\phi: \M_{\tau} \to \M(\Sigma,\sigma,\tau)$ defined by
\begin{equation*}
\mu \mapsto \frac{\mu \times Leb}{\int \tau d\mu},
\end{equation*}
where $Leb$ is the one dimensional Lebesgue measure, is a bijection.
\end{lemma}
We denote  the inverse of $\phi$ by $\psi$.  We will be particularly interested in a special class of roof functions.

 \begin{definition} \label{def:R} A positive function $\tau:\Sigma\to\R$ belongs to the class $\cR$ if the following properties hold: 
\begin{enumerate}
\item \label{Ta} $\tau$ is uniformly continuous, bounded away from zero, and $var_2(\tau)$ is finite,
\item \label{Tb} $$\lim_{k\to \infty}\inf_{x: x_1 \ge k}\tau(x)=\infty,$$
where $x_1$ is the first coordinate of $x$. 
\end{enumerate}
\end{definition}

\begin{remark} The class $\cR$ includes a wealth of interesting examples. For instance, the geodesic flow over the modular surface can be coded as a suspension flow over the full-shift on a countable alphabet, $\Sigma=\N^\N$, with a roof function $\tau$ belonging to $\cR$, see \cite{gk,ku} for details. A large class of examples belonging to the  class $\cR$ is to be found in one-dimensional dynamics. Indeed, the class of Expanding-Markov-Renyi  (EMR) maps is a class of interval maps with infinitely many branches which was introduced by Pollicott and Weiss in \cite{pw} and has been extensively studied. It turns out that if $f$ is an EMR map then  the symbolic version of the corresponding  geometric potential $ \log |f'|$ belongs to $\cR$.  These potentials carry the relevant fractal information of the system as well as the coding of relevant equilibrium measures such as  Sinai-Ruelle-Bowen measures. An example of an EMR map is the Gauss map.
\end{remark}

\subsection{The topology of convergence on cylinders} \label{subsec:ccs}
The space of invariant sub-probability measures of the suspension flow is denoted by $\M_{\le 1}(\Sigma,\sigma,\tau)$. In this section we endow this space with a topology that makes it compact whenever $\tau\in \cR$ (see Theorem \ref{compact}). The topology we consider is an adaptation of the cylinder topology defined in sub-section \ref{sec:cyl}. Let $(Y, \Theta)$ be a suspension flow over $(\Sigma, \sigma)$ with roof function $\tau$ and $c=c(\tau)>0$ such that  $\inf \tau>c$.
\begin{definition} \label{def:ccs} Let $(\nu_n)_n$ and $\nu$ be measures in $\M_{\le 1}(\Sigma,\sigma,\tau)$.  We say that $(\nu_n)_n$ \emph{converges on cylinders} to $\nu$ if $$\lim_{n\to\infty}\nu_n(C\times [0,c])=\nu(C\times[0,c]),$$
for every cylinder $C\subset \Sigma$.
\end{definition}

Recall that by Kac's formula we know that $$\nu(C\times [0,c])=\frac{c\mu(C)}{\int \tau d\mu},$$ whenever $\nu\in \M(\Sigma,\sigma,\tau)$ and $\mu=\psi(\nu)$. 

\begin{remark}\label{rem:tobase} Let $\nu$ and $(\nu_n)_n$ be invariant probability measures for the suspension flow and set $\mu_n=\psi(\nu_n)$, and $\mu=\psi(\nu)$. From the definition of $\psi$  we have that the following statements are equivalent
\begin{enumerate}
\item The sequence $(\nu_n)_n$ converges on cylinders to $\lambda \nu$, where $\lambda\in [0,1]$. 
\item The following limit holds $$\lim_{n\to\infty}\frac{\mu_n(C)}{\int \tau d\mu_n}=\lambda\frac{\mu(C)}{\int \tau d\mu},$$
for every cylinder $C\subset \Sigma$.
\end{enumerate}
Moreover, $\lambda_1\nu_1$ and $\lambda_2\nu_2$ are equal if and only if $$\lambda_1\nu_1(C\times [a,b])=\lambda_2\nu_2(C\times [a,b]),$$ for every cylinder $C\subset \Sigma$ and $a,b\in\R$.  By Kac's formula this is equivalent to 
$$\lambda_1\frac{\mu_1(C)}{\int \tau d\mu_1}=\lambda_2\frac{\mu_2(C)}{\int \tau d\mu_2},$$
for every cylinder $C\subset \Sigma$. 
\end{remark}

\begin{lemma} \label{met}
The topology of the convergence on cylinders in $\M_{\le 1}(\Sigma,\sigma,\tau)$ is metrizable.
\end{lemma}

\begin{proof}

Let  $\rho: \M_{\le 1}(\Sigma,\sigma,\tau) \times \M_{\le 1}(\Sigma,\sigma,\tau) \to \R$, be defined by
\begin{equation*}
\rho(\nu_1,\nu_2)=\sum_{k\ge 1}\frac{1}{2^k} \left|\nu_1(C_i\times [0,c])-\nu_2(C_i\times[0,c]) \right|,
\end{equation*}
where $(C_i)_i$ is some enumeration of the cylinders of $\Sigma$ and $c=c(\tau)$. The map $\rho$ is a metric. Indeed, let $\nu_1=\lambda_1 \phi(\mu_1)$ and $\nu_2=\lambda_2\phi(\mu_2)$ be in $\M_{\le 1}(\Sigma,\sigma,\tau)$, where $\lambda_1$ and $\lambda_2$ are in $[0,1]$.  Suppose that $\rho(\nu_1,\nu_2)=0$. By Remark \ref{rem:tobase} we know that 
 $$\lambda_1\frac{\mu_1(C)}{\int \tau d\mu_1}=\lambda_2\frac{\mu_2(C)}{\int \tau d\mu_2},$$
for every cylinder $C\subset \Sigma$. If $\lambda_1=0$ we necessarily have $\lambda_2=0$: in this case $\nu_1$ and $\nu_2$ are both the zero measure. Assume that $\lambda_1\ne 0$, then
$$\mu_1(C)=\frac{\lambda_2}{\lambda_1}\frac{\int \tau d\mu_1}{\int \tau d\mu_2}\mu_2(C),$$
for every cylinder $C\subset \Sigma$. By the outer regularity of Borel measures on a metric space we conclude that $\mu_1=A\mu_2$, where $A=\frac{\lambda_2}{\lambda_1}\frac{\int \tau d\mu_1}{\int \tau d\mu_2}$. This immediately implies that $\lambda_1\frac{\mu_1}{\int\tau d\mu_1}=\lambda_2\frac{\mu_2}{\int \tau  d \mu_2}$, and therefore $\nu_1=\nu_2$.  The other properties of a metric are easily verified. Note that $(\nu_n)_n$ converges on cylinders to $\nu$ if and only if $\lim_{n\to\infty}\rho(\nu_n,\nu)=0$, that is, the topology of convergence on cylinders is metrizable.  
\end{proof}

Our next result should be compared with Lemma \ref{equivtop}. It says that the topology of convergence on cylinders coincides with the weak* topology on $\M(\Sigma,\sigma,\tau)$. We emphasize that this result holds for every $\tau$ which is bounded below. 

\begin{lemma} \label{lem:w.c} Let $\nu$ and $(\nu_n)_n$ be measures in $\M(\Sigma,\sigma,\tau)$. The following assertions are equivalent
\begin{enumerate}
\item\label{1wc} The sequence $(\nu_n)_n$ converges in the weak* topology to $\nu$.
\item\label{2wc} The sequence $(\nu_n)_n$ converges on cylinders to $\nu$. 
\end{enumerate}
\end{lemma}


\begin{proof} Let $C\subset \Sigma$ be a cylinder. Observe that $\partial(C\times [0,c])=C\times\{0,c\}$. Since $\nu$ is a flow invariant probability measure we know that $\nu(C\times \{x\})=0$, for each $x\in \R$. We conclude that $\nu(\partial(C\times [0,c]))=0$. Finally use  Proposition \ref{port}(\ref{port4}) to conclude that (\ref{1wc}) implies (\ref{2wc}). 

We will now prove that (\ref{2wc}) implies (\ref{1wc}). A base for the topology in $Y$ is given by 
\begin{equation*}
\Omega:=\{C \times (a,b) \subset Y : C \text{ cylinder for	} \Sigma \text{ and }  a,b \in \mathbb{Q} \text{ with }  a<b \}.
\end{equation*}
It follows from the flow invariance of the measures that for every set $ C \times (a,b) \subset Y$ we have
\begin{equation*}
\lim_{n \to \infty} \nu_n( C \times (a,b) ) =  \nu(C \times (a,b)).
\end{equation*}
Observe that a finite intersection of elements in $\Omega$ is still in $\Omega$. Note that each open set $\cO\subset Y$ can be written as a countable union of elements in $\Omega$, say $\cO=\bigcup_{k\ge 1}(C_k \times (a_k, b_k))$. The result now follows from  \cite[Theorem 2.2]{bi}.
\end{proof}

One of the main properties of the class $\cR$ is that we can rule out the escape of mass by imposing a uniform bound on the integral of $\tau$ (see Lemma \ref{lem:noescape}). This illustrates the importance of part \eqref{Tb}  in  Definition \ref{def:R}.  

\begin{lemma}\label{lem:noescape} 

Let $(\mu_n)_n \subset \M(\Sigma,\sigma)$  and $\mu \in \M_{\leq 1}(\Sigma,\sigma)$ be such that $(\mu_n)_n$ converges on cylinders to the measure $\mu$. Let $\tau\in \cR$ and assume there exists $M\in \R$ such that $\int \tau d\mu_n\le M$, for every $n\in \N$. Then $\mu$ is a probability measure. Moreover, $(\mu_n)_n$ converges to $\mu$ in the weak* topology. 
\end{lemma}

\begin{proof} Observe that for every $k \in \N$ we have
\begin{equation*}
 \left(\inf_{x: x_0\ge k}\tau(x)\right) \mu_n\left(\bigcup_{s\ge k}[s]\right)\le \int \tau d\mu_n\le M, 
\end{equation*}
then 
\begin{equation*}
 \mu_n \left(\bigcup_{s\ge k}[s]\right)\le \frac{M}{\inf_{x: x_0\ge k}\tau(x)}.
\end{equation*} 
This is equivalent  to $\mu_n\big(\bigcup_{s< k}[s]\big)\ge 1-\frac{M}{\inf_{x: x_0\ge k}\tau(x)}.$ By definition of the convergence on cylinders we have 
\begin{equation*}
\mu\left(\bigcup_{s< k}[s]\right)=\lim_{n\to\infty}\mu_n\left(\bigcup_{s< k}[s]\right)\ge 1-\frac{M}{\inf_{x: x_0\ge k}\tau(x)}.
\end{equation*}
Since $\tau\in\cR$ we can conclude that $\lim_{k\to\infty} \mu\big(\bigcup_{s< k}[s]\big)=1$, and therefore $\mu$ is a probability measure. Since the sequence $(\mu_n)_n$ converges on cylinders to a probability measure we conclude that $(\mu_n)_n$ converges in the weak* topology  (see Lemma \ref{equivtop}). 
\end{proof}

Our next two lemmas completely describe the topology of convergence on cylinders in terms of convergence of measures in $\Sigma$. 

\begin{lemma}\label{lem:charzero} Let $\tau\in \cR$. A sequence $(\nu_n)_n\subset \M(\Sigma,\sigma,\tau)$   converges on cylinders to the zero measure if and only if 
\begin{equation*}
\lim_{n\to\infty} \int \tau d\psi(\nu_n)=\infty.
\end{equation*}
\end{lemma}

\begin{proof} To simplify notation define $\mu_n=\psi(\nu_n)$. We will first prove that 
\begin{equation} \label{eq:i}
\lim_{n\to\infty} \int \tau d\mu_n=\infty,
\end{equation}
implies that $(\nu_n)_n$ converges on cylinders to the zero measure. Note that equation \eqref{eq:i} implies that
\begin{equation*}
 \lim_{n\to\infty}\frac{\mu_n(C)}{\int \tau d\mu_n}=0,
 \end{equation*}
for every cylinder $C$. In virtue of  Remark \ref{rem:tobase} we get that $(\nu_n)_n$ converges on cylinders to the zero measure. To prove the other implication we argue by contradiction: suppose that the sequence $(\nu_n)_n$ converges to the zero measure and that there exists a subsequence $(n_k)_k$ such that $\int \tau d\mu_{n_k}\le M$, for some $M\in\R$. From Remark \ref{rem:tobase} we obtain
\begin{equation*}
 0=\lim_{k\to\infty} \frac{\mu_{n_k}(C)}{\int \tau d\mu_{n_k}}\ge\frac{1}{M} \limsup_{k\to\infty}\mu_{n_k}(C).
 \end{equation*}
In particular, $(\mu_{n_k})_k$ converges on cylinders to the zero measure. Lemma \ref{lem:noescape} implies that $\lim_{k \to \infty}\int \tau d\mu_{n_k}=\infty$, which contradicts the choice of the sequence $(n_k)_k$. 
\end{proof}

\begin{lemma}\label{lem:util} Let $\tau\in\cR$ and $(\nu_n)_n$, $\nu$ invariant probability measures for the suspension flow. Define $\mu_n=\psi(\nu_n)$ and $\mu=\psi(\nu)$. Then the following are equivalent:
\begin{enumerate}
\item\label{a} The sequence $(\nu_n)_n$  converges on cylinders to $\lambda\nu$, where $\nu\in \M(\Sigma,\sigma,\tau)$ and $\lambda\in (0,1]$.
\item\label{b} The sequence $(\mu_n)_n$ converges to $\mu$ in the weak* topology and $$\lim_{n\to\infty}\frac{\int \tau d\mu}{\int \tau d\mu_n}=\lambda\in (0,1].$$
\end{enumerate}
\end{lemma}
\begin{proof} We first prove that (\ref{b}) implies (\ref{a}). If $(\mu_n)_n$ converges in the weak* topology to $\mu$, then $\lim_{n\to\infty}\mu_n(C)=\mu(C)$, for every cylinder $C$. It follows from the hypothesis on $\lambda$ that  
$$\lim_{n\to\infty} \frac{\mu_n(C)}{\int \tau d\mu_n}=\lambda\frac{\mu(C)}{\int \tau d\mu}.$$
Remark \ref{rem:tobase} implies that $(\nu_n)_n$ converges on cylinders to $\lambda\nu$.  

Now suppose that $(\nu_n)_n$ converges on cylinders to $\lambda\nu$. It follows from Lemma \ref{lem:charzero} and the assumption that $\lambda >0$ that 
 $(\int \tau d\mu_n)_n$ is a bounded sequence.  After passing to a subsequence we can assume that $(\int \tau d\mu_n)_n$ is convergent. Let $L \in \R$ be such that $\lim_{n \to \infty}\int \tau d\mu_n=L$. Remark \ref{rem:tobase} implies that 
$$\lim_{n\to\infty} \mu_n(C)=\frac{\lambda L}{\int\tau d\mu}\mu(C),$$
for every cylinder $C$. We conclude that $(\mu_n)_n$ converges on cylinders to $\mu_0:=\frac{\lambda L}{\int\tau d\mu}\mu$. Observe that for sufficiently large $n$ we have $\int \tau d\mu_n\le (L+1).$ Lemma \ref{lem:noescape} implies that $\mu_0$ is a probability measure and that $(\mu_n)_n$ converges in the weak*  topology for $\mu_0$. Since $\mu$ is a probability measure we conclude that $\lambda L=\int \tau d\mu$, and therefore $\mu_0=\mu$. This argument shows that every subsequence of our initial sequence $(\mu_n)_n$ has a sub-subsequence converging to $\mu$. This readily implies that the whole sequence converges to $\mu$, and that $\lim_{n\to\infty}\frac{\int \tau d\mu}{\int \tau d\mu_n}=\lambda$.
\end{proof}



Our next result should be compared with Lemma \ref{tozero}. As mentioned in the introduction, this is a necessary ingredient to prove that $\M_{\le1}(\Sigma,\sigma,\tau)$ is the Poulsen simplex. 

\begin{lemma}  \label{ftz}
If $\tau\in \cR$ then there exists a sequence of periodic measures $(\nu_n)_n\subset \M(\Sigma,\sigma,\tau)$ that converges on cylinders to the zero measure.
\end{lemma}

\begin{proof}
We will separate our analysis into two cases. 

\emph{Case 1 (assume $(\Sigma,\sigma)$ satisfies the $\F-$property)}: By Lemma \ref{tozero}  there exists a sequence of periodic measures $(\mu_n)_n$ which converges on cylinders to the zero measure. Observe that every periodic measure belong to $\M_\tau$, in particular   $\phi(\mu_n) \in \M(\Sigma,\sigma,\tau)$. Now, by Lemma \ref{lem:noescape} we conclude that $\lim_{n\to\infty}\int \tau d\mu_n=\infty$. It follows from Lemma \ref{lem:charzero} that the sequence $(\phi(\mu_n))_n$ converges to the zero measure.

\emph{Case 2 (assume $(\Sigma,\sigma)$ does not satisfies the $\F-$property}): In this case there exists an element $a$ in the alphabet, $l\in \N$, and a sequence $(p_n)_n$ of distinct periodic points of length $l$ such that $p_n\in [a]$. The periodic measure associated to $p_n$ is denoted by $\eta_n$. Since $\tau \in \cR$, for $N\in\R$ there exists $n_N$ such that for $n\ge n_N$ we have $S_l\tau(p_{n})>N$. Here $S_l\tau$ is the Birkhoff sum of $\tau$ of length $l$. In particular  $\int \tau d\eta_n=\frac{1}{l}S_l\tau(p_{n})\ge N/l$. This implies that $\lim_{n\to\infty}\int \tau d\eta_n=\infty$. The result then  follows from Lemma \ref{lem:charzero}.
\end{proof}


Recall that the set $L(\Sigma)$ was introduced in Definition \ref{def:L}. We will now prove a compactness result similar to Proposition \ref{extension}. The proof of Proposition \ref{prop:F} is significantly simpler than the one of Proposition \ref{extension}; it would be interesting if this result can be generalized to a larger class of potentials. 

\begin{proposition}\label{prop:F} Assume that $\tau\in \cR$. Let $(\mu_n)_n$ be a sequence of periodic measures on $\Sigma$. Suppose that $(\mu_n)_n$ converges on cylinders to $F\in L(\Sigma)$, and that there exists $M\in \R$ such that $\int \tau d\mu_n\le M$, for all $n\in \N$. Then $F$ extends to an invariant probability measure. 
\end{proposition}

\begin{proof} We will follow the strategy of the proof of Proposition \ref{extension}. It is enough to prove that  
\begin{equation*}
\lim_{k\to\infty}\lim_{n\to\infty}\mu_n\left(\bigcup_{s\ge k} Cs\right)=0,
\end{equation*}
where $C=[a_0...a_{m-1}]$ is a cylinder.   Let $p_n\in \Sigma$ be a periodic point of period $r_n$  and  $\mu_n$ be the periodic measure  associated to the point $p_n$. It is important to observe that 
\begin{align}\label{xx}
\mu_n\left(\bigcup_{s\ge k}[s]\right)\ge \mu_n\left(\bigcup_{s\ge k} Cs\right).
\end{align}
Indeed, the probability measure $\mu_n$ is equidistributed on the set 
\begin{equation*}
\{p_n,\sigma(p_n),...,\sigma^{r_n-1}(p_n)\}.
\end{equation*}
Observe that $\sigma^k(p_n)\in Cs$, implies that $\sigma^{k+m}(p_n)\in [s]$, from which inequality (\ref{xx}) follows. Since $\int \tau d\mu_n\le M$, we obtain that 
\begin{equation*}
\mu_n\left(\bigcup_{s\ge k}[s]\right)\le \frac{\int \tau d\mu_n}{\inf_{x: x_1\ge k}\tau(x)}\le\frac{M}{\inf_{x: x_1\ge k}\tau(x) },
\end{equation*}
but this immediately implies that 
$$\lim_{k\to\infty}\lim_{n\to\infty} \mu_n\left(\bigcup_{s\ge k}[s]\right)\le \limsup_{k\to\infty}\frac{M}{\inf_{x: x_1\ge k}\tau(x) }.$$
Since $\tau\in \cR$ we obtain that the right hand side in the last inequality is zero. Finally use inequality (\ref{xx}) to conclude that 
\begin{equation*}
\lim_{k\to\infty}\lim_{n\to\infty}\mu_n\left(\bigcup_{s\ge k} Cs\right)=0.
\end{equation*}
 As in the proof of Proposition \ref{extension} we have that $F$ extends to a measure on $\Sigma$. Lemma \ref{lem:noescape} implies that $F$ is a probability measure. The invariance follows from Lemma \ref{invariance}. 
\end{proof}

Proposition \ref{prop:F} states that, assuming an integrability condition, limits of periodic measures are invariant  probability measures. It is then of particular importance to know whether it  is possible to approximate a sequence of invariant measures  $(\mu_n)_n$ by periodic measures such that the assumption $\sup_n\int \tau d\mu_n<\infty$ still remains true for the sequence of periodic measures. Proposition \ref{densusp} address  this question. In the proof of Proposition \ref{densusp} we will need to approximate a measure in $\M(\Sigma,\sigma)$ by a convex combination of finitely many ergodic probability measures. This result is  classical in the compact case: the space of invariant probability measures is a compact convex set, and therefore the result follows from the  Krein-Milman theorem. In lack of a good reference we provide a proof of this result that avoids the Krein-Milman theorem. 

\begin{lemma}\label{noref} Let $\mu\in \M(\Sigma,\sigma)$ and $(f_i)_{i=0}^n$ be real-valued functions in $L^1(\mu)$. Given $\epsilon>0$, there exists  $\mu_1\in\M(\Sigma,\sigma)$ that is a convex combination of finitely many ergodic probability measures which satisfies
$$\left|\int f_i d\mu-\int f_i d\mu_1 \right|\le\epsilon,$$
for every $i\in\{0,...,n\}$.
\end{lemma}

\begin{proof} It is enough to prove the result under the assumption that each $f_i$ is non-negative. Indeed, let $f_{i}^+$ and $f_{i}^-$ be the positive and negative parts of $f_i$. Applying the result to the set  $(f_i^+)_i\bigcup(f_i^-)_i$  and $\epsilon/2$, we obtain  a measure $\mu_1$. By triangle inequality, the measure $\mu_1$ verifies the inequalities in the statement of Lemma \ref{noref}.

From now on we assume that each $f_i$ is non-negative. Let $\mu=\int \mu_x d m(x)$ be the ergodic decomposition of $\mu$, in particular, 
\begin{equation*}
\int f_k d\mu=\int \left(\int f_k d\mu_x \right)d m(x),
\end{equation*}
for $k\in \{0,...,n\}$.   By ergodic decomposition, the measure $\mu_x$ is ergodic for  $\mu$-almost every $x\in \Sigma$. We choose  a measurable set  $S\subset \Sigma$ such that $\mu(S)=1$, $\mu_x$ is ergodic for every $x\in S$, and $f_k$ is $\mu_x$-integrable for every $x\in S$ and $k\in\{0,...,n\}$. Given $x\in S$ we define $F_k(x):=\int f_k d\mu_x$. Observe that $\int f_kd\mu=\int_S F_k d\mu$. By definition of the integral we know that 
\begin{align*} 
\int_S F_k d\mu
=\sup \left\{\int_S g d\mu: g\text{ is simple and }g\le F_k \right\}.
\end{align*}
This immediately implies that there exists a measurable partition $\cU_k=\{U^{(k)}_1,...,U_{p_k}^{(k)}\}$ of $S$, such that, given $\epsilon >0$, 
\begin{align*}
\int_S F_k d\mu-\epsilon &<\sum_{i=1}^{p_k} \left(\inf_{x\in U_i^{(k)}} F_k(x)\right)\mu(U_i^{(k)}) 
\le  \int_S F_kd\mu
\end{align*}
If $\cP$ and $\cQ$ are partitions then $\cP\wedge\cQ$   denotes the common refinement of $\cP$ and $\cQ$. That is,  $A\in \cP\wedge \cQ$ if $A=P_1\cap Q_1$, where $P_1\in \cP$ and $Q_1\in \cQ$. Consider the partition $\cU=\bigwedge_{i=0}^n \cU_i$ of $S$, and write $\cU=\{U_1,...,U_q\}$. Choose a point $x_i\in U_i$, for every $i\in \{1,...,q\}$. Note that 
\begin{equation*}
\int_S F_k d\mu-\epsilon<\sum_{i=1}^{q} F_k(x_i)\mu(U_i)\le \int_S F_kd\mu,
\end{equation*}
for every $k\in \{0,...,n\}$. We conclude that 
\begin{equation*}
\left|\int f_k d\mu-\sum_{i=1}^q \mu(U_i)\int f_k d\mu_{x_i} \right|\le \epsilon,
\end{equation*}
for every $k\in \{0,...,n\}$. Finally define $\mu_1=\sum_{i=1}^q \mu(U_i)\mu_{x_i}$.

\end{proof}

We will now prove a refinement of Theorem \ref{lem:cs}. As mentioned before, this is an important ingredient to increase the applicability of Proposition \ref{prop:F}. Recall that $\tau$ is positive and bounded away from zero.

\begin{proposition} \label{densusp} Let $(\Sigma, \sigma)$ be a transitive countable Markov shift and  $\tau$ be a uniformly continuous function such that $var_2(\tau)$ is finite. For every  $\mu\in \M_\tau$ there exists a sequence of periodic measures $(\mu_n)_n$ that converges in the weak* topology to $\mu$  and such that  $\lim_{n\to\infty}\int \tau d\mu_n=\int \tau d\mu$. 
\end{proposition}

\begin{proof}  In virtue of Lemma \ref{noref} it is possible to find a sequence $(\mu_n)_n$ of measures, each a convex combination of finitely many ergodic probability measures, satisfying $d(\mu_n,\mu)\le \frac{1}{n}$, and such that $\lim_{n\to\infty}\int\tau d\mu_n=\int \tau d\mu$.  In particular, it is enough to prove that the result holds for measures which are a finite convex combination of ergodic measures. We can moreover assume that the weights in the convex combination are rational numbers. Thus, from now on we assume that   $\mu=\frac{1}{N}\sum_{j=1}^N \mu_j$, where each $\mu_j$ is ergodic. 

Let $\F_0=\{f_1,...,f_l\}\subset C_b(\Sigma)$ be a collection of bounded uniformly continuous functions on $\Sigma$  and define $\F:=\F_0\cup \{\tau\}$. By assumption each $f\in \F$ is uniformly continuous. In particular, given $\epsilon>0$, there exists $N_0=N_0(\epsilon)$ such that $var_n(f)\le \frac{\epsilon}{4}$, for every $f\in \F$ and $n\ge N_0$. Define $C_0:=\max_{f\in \F_0}\max_{x\in \Sigma}|f(x)|$. 

Choose $M$ such that $\mu_j(K_M)>9/10$, for every $j\in \{1,...,N\}$, where $K_M=\bigcup_{s=1}^M[s]$. By transitivity of $(\Sigma,\sigma)$ there exists a number $L$ such that every pair of numbers in $\{1,...,M\}^2$  can be connected with an admissible word of length at most $L$. For each pair $(a,b)\in\{1,...,M\}^2$ we choose a point $p_{a,b}$ such that $p_{a,b}\in[a]$ and $\sigma^{c(a,b)-1}(p_{a,b})\in [b]$, where $c(a,b)\le L$. Recall that $S_n\tau(x)$ denotes the Birkhoff sum of length $n$ of the point $x$. Set $C_1=\max_{a,b}|S_{c(a,b)}\tau(p_{a,b})|$.  Define 
\begin{equation*}
A^s_{j,\epsilon}= \left\{ x\in \Sigma: \left|\frac{1}{m}\sum_{i=0}^{m-1}f(\sigma^i x)-\int f d\mu_j \right|<\frac{\epsilon}{4},\text{ for every }f\in\F \text{ and }m\ge s \right\}.
\end{equation*}
It follows from Birkhoff ergodic theorem that $\mu_j(A_{j,\epsilon}^s)\to 1$ as $s\to\infty$. Choose $s_0 \in \N$ such that $\mu_j(A_{j,\epsilon}^{s_0})\ge 9/10$, for every $j \in  \{1, \dots,  N  \}$. We assume that $s_0$ is sufficiently large (relative to our constants $C_0$, $C_1$, $N_0$ and $L$) to be determined later.

Observe that $\mu_j( A^{s_0}_{j,\epsilon}\cap K_M\cap \sigma^{-s_0}(K_M))>\frac{1}{2}$. Pick a point $x_j\in A^{s_0}_{j,\epsilon}\cap K_M\cap \sigma^{-s_0}(K_M)$. We will construct a periodic point $x_0$ out of the sequence $(x_j)^N_{j=1}$. Let ${\bf y_j}$ be the admissible word coming from the first $(s_0+1)$ coordinates of $x_j$. Observe that the first and last letters of ${\bf y_j}$ are in $\{1,...,M\}$. We construct an admissible word of the form ${\bf y}={\bf y_1 w_1y_2w_2...y_Nw_N}$, where ${\bf w_i}$ are admissible words of length  less or equal to $L$ that connects ${\bf y_i}$ with ${\bf y_{i+1}}$ (where we consider ${\bf y_{N+1}}={\bf y_1}$). We will moreover assume that the admissible word ${\bf w_i}$ is the same one we used to construct the point $p_{a,b}$, for the corresponding $a$ and $b$. In this case $l({\bf y_i})=c(a,b)$ and the point associated to ${\bf w_i}$ is denoted by $p_i\in (p_{a,b})_{a,b}$. Then define $x=({\bf yy...})$.  We claim that the periodic measure associated to $x$, say $\mu_x$, belongs to the set 
\begin{equation*}
\Omega= \left\{ \nu \in \M(\Sigma,\sigma): \left|\int fd\nu-\int fd\mu \right|<\epsilon, \text{ for every }f\in \F \right\}.
\end{equation*}
Our construction ensures the following inequalities:
\begin{enumerate}
\item $|S_{s_0-N_0}f(x)-S_{s_0-N_0}f(x_1)|\le (s_0-N_0)var_{N_0}f\le (s_0-N_0)\epsilon/4,$ for every $f\in \F$. 
\item $|S_{N_0+l({\bf y_1})}f(\sigma^{s_0-N_0}x)-S_{N_0+l({\bf y_1})}f(\sigma^{s_0-N_0}x)|\le 2(N_0+l({\bf y_1}))C_0\le 2(N_0+L)C_0,$ for every $f\in \F_0$. 
\item $|S_{N_0}\tau(\sigma^{s_0-N_0}x)-S_{N_0}\tau(\sigma^{s_0-N_0}x_1)|\le N_0 var_2(\tau)$.  
\item $|S_{l({\bf y_1})}\tau(\sigma^{s_0}x)-S_{l({\bf y_1})}\tau(p_1)|\le l({\bf y_1})var_2(\tau)\le Lvar_2(\tau).$
\end{enumerate}
 We can use the last inequality to obtain that 
 \begin{equation*}
|S_{l({\bf y_1})}\tau(\sigma^{s_0}x)|\le Lvar_2(\tau)+C_1.
\end{equation*}
Combining these inequalities we obtain that 
\begin{equation*}
|S_{s_0+l({\bf y_1})}\tau(x)-S_{s_0}\tau(x_1)|\le \frac{1}{4}(s_0-N_0)\epsilon+N_0var_2(\tau)+Lvar_2(\tau)+C_1,
\end{equation*}
and that 
\begin{equation*}
|S_{s_0+l({\bf y_1})}f(x)-S_{s_0+l({\bf y_1})}f(x_1)|\le \frac{1}{4}(s_0-N_0)\epsilon+2(N_0+L)C_0,
\end{equation*}
where $f\in \F_0$. Similar inequalities can be obtained when comparing the value of our function $f$ at $\sigma^{(k-1)s_0+\sum_{i=1}^{k-1}l({\bf y_i})}(x)$ and at $x_k$, where $f\in \F$. Using the triangle inequality and the definition of $ A^{s_0}_{j,\epsilon}$ we can estimate $|\int f d\mu_x-\sum_{j=1}^N\int fd\mu_j|,$ in an effective way. By taking $s_0$ large enough  (in terms of our constants $C_0$, $C_1$, $N_0$ and $L$)  we can ensure that $\mu_x\in \Omega.$ We leave the details to the reader. We have now proved that the result holds for convex combinations of ergodic measures, and as explained at the beginning of the proof, the general result follows from Lemma \ref{noref}.
\end{proof}


\begin{corollary}\label{cor:densusp} Assume that $\tau \in \cR$. Then the space of ergodic measures $\M_e(\Sigma,\sigma,\tau)$ is  weak* dense in $\M(\Sigma,\sigma,\tau)$.
\end{corollary}
\begin{proof} Fix $\nu\in \M(\Sigma,\sigma,\tau)$. We will prove that $\nu$ can be approximated in the topology of convergence on cylinders by ergodic measures. Let $\mu=\psi(\nu)$. By Proposition \ref{densusp} we can find a sequence of periodic measures $(\mu_n)_n\subset \M(\Sigma,\sigma)$ converging to $\mu$ in the weak* topology and such that $\lim_{n\to\infty}\int \tau d\mu_n=\int \tau d\mu$. Set $\nu_n=\phi(\mu_n)$. Lemma \ref{lem:util} implies that $(\nu_n)_n$ converges in the cylinder topology to $\nu$. In virtue of Lemma  \ref{lem:w.c} we have that  $(\nu_n)_n$ converges in the weak*  topology to $\nu$. Moreover, each $\nu_n$ is ergodic, since each $\mu_n$ is ergodic. This concludes the proof of the corollary. 
\end{proof}

We have finally all the ingredients to prove the main result of this section: the compactness of $\M_{\le 1}(\Sigma,\sigma,\tau)$. We already know that $\M_{\le1}(\Sigma,\sigma,\tau)$ is a metrizable topological space, it is enough to prove it is sequentially compact.

\begin{theorem} \label{compact} Assume that $\tau\in \cR$. Let $(\nu_n)_n$ be a sequence of invariant probability measures of the suspension flow. Then there exists a subsequence $(\nu_{n_k})_k$ converging on cylinders to an invariant sub-probability measure $\nu$. 
\end{theorem}

\begin{proof} Let $\mu_n=\psi(\nu_n)$. If $\limsup_{n\to\infty}\int \tau d\mu_n=\infty$, there exists  a subsequence of $(\nu_n)_n$ that converges on cylinders to the zero measure. We will assume that $$\int \tau d\mu_n\le M,$$ for all $n\in \N$. By compactness of $L(\Sigma)$ there exists a subsequence $(\mu_{n_k})_k$ that converges on cylinders to $F\in L(\Sigma)$. Maybe after passing to a subsequence we can  assume that $\lim_{k\to\infty}\int \tau d\mu_{n_k}=L,$ for some $L\in \R$.  We can now use Proposition \ref{densusp} to obtain periodic measures $\eta_k$ satisfying $d(\mu_{n_k},\eta_k)\le \frac{1}{k}$ and $\int \tau d\eta_k\le (\int \tau d\mu_{n_k}+1)$. Note that $\lim_{k\to\infty}d(\eta_k,F)=0$ and  $\int \tau d\eta_k\le (M+1)$, for all $k\in \N$. We can now use Proposition \ref{prop:F} to conclude that $F$ extends to an invariant probability measure that we denote by $\mu$. It follows that $(\mu_{n_k})_k$ converges on cylinders to $\mu$ and that $\lim_{k\to\infty}\int \tau d\mu_{n_k}=L$. Finally use Lemma \ref{lem:util} to obtain that $(\nu_{n_k})_k$ converges on cylinders to $\lambda \nu$, where $\nu=\phi(\mu)$ and $\lambda=\frac{\int \tau d\mu}{L}$. 
\end{proof}

\subsection{The space of flow invariant sub-probability measures is the Poulsen simplex}
In Section \ref{sec:pou} we proved that $\M_{\le1}(\Sigma,\sigma)$ is affine homeomorphic to the Poulsen simplex if $(\Sigma,\sigma)$ has the $\F-$property. In this section we  prove an analogous result for the suspension flow, that is, that $\M_{\le1}(\Sigma,\sigma,\tau)$ is affine homeomorphic to the Poulsen simplex, provided that $\tau\in \cR$.

\begin{theorem} \label{susp:p1}
Let $\tau$ be a potential in $\cR$. The space $\M_{\le 1}(\Sigma,\sigma,\tau)$ endowed with the topology of convergence on cylinders is affine homeomorphic to the Poulsen simplex. 
\end{theorem}

\begin{proof}
In Theorem \ref{compact} we proved that the space $\M_{\le 1}(\Sigma,\sigma,\tau)$ is compact with respect to the cylinder topology. In Lemma  \ref{met} we showed that it is a metrizable space. Since the space is also a convex Choquet simplex (from the ergodic decomposition), it suffices to prove that the set of extreme points is dense. Note that every element of $\M_{\le 1}(\Sigma,\sigma,\tau)$ is of the form $\lambda \nu$, with $\lambda \in [0,1]$ and $\nu \in \M(\Sigma,\sigma,\tau)$. In Lemma \ref{ftz} we proved that there exists a sequence of flow invariant ergodic measures $(\tilde{\nu}_n)_n$ converging to the zero measure. Set $\hat{\nu}_n:= \lambda \nu + (1- \lambda)\tilde{\nu}_n$, and observe that the sequence $(\hat{\nu}_n)_n$ converges on cylinders to $\lambda \nu$. It follows from Corollary \ref{cor:densusp} that  the ergodic measures are dense in $\M(\Sigma,\sigma,\tau)$. This allows us to approximate $(\hat{\nu}_n)_n$ by a sequence of ergodic measures that converges in cylinders to $\lambda \nu$. It follows from the main result of \cite[Theorem 2.3]{los} that $\M_{\le1}(\Sigma,\sigma,\tau)$ is affine homeomorphic to the Poulsen simplex. 
 \end{proof}

As in Section \ref{sec:pou} we conclude that 


\begin{theorem} \label{susp:p2}
Let $(\Sigma, \sigma)$ be a transitive countable Markov shift  and $\tau \in \cR$. Then  $\M(\Sigma,\sigma,\tau)$ is affinely homeomorphic to the Poulsen simplex minus a vertex and all of its convex combinations.
\end{theorem}

\begin{proof}
Note that the set of extreme points of $\M_{\le 1}(\Sigma,\sigma,\tau)$ is the zero measure together with the set  of ergodic measures in $\M(\Sigma,\sigma,\tau)$.
The result now follows from Theorem \ref{susp:p1} together with the relation between weak* and cylinder topology (see Lemma \ref{lem:w.c}).
\end{proof}

\end{document}